\documentclass{amsart}
\usepackage{amsmath}
\usepackage{amssymb}
\usepackage{mathrsfs}
\usepackage{pst-node}
\usepackage{tikz-cd}
\usepackage{comment}
\usetikzlibrary{shapes.geometric}
\usepackage{enumerate}
\usepackage{amsthm}
\usepackage{enumitem}
\usepackage{fullpage}
\usepackage{euscript}
\usepackage{lettrine}
\usepackage{graphicx}
\usepackage{mathtools}
\usepackage[titletoc]{appendix}
\usepackage{MnSymbol}
\usepackage{GoudyIn}
\definecolor{darkblue}{rgb}{0,0,0.6} 
\usepackage{hyperref}

\usepackage[bbgreekl]{mathbbol}

\hypersetup{
    colorlinks=true,
    citecolor=darkblue,
    filecolor=darkblue,
    linkcolor=darkblue,
    urlcolor=darkblue
}

\usepackage[margin=1.25in,headheight=110pt]{geometry}

\theoremstyle{plain}
\newtheorem{theorem}{Theorem}[section]
\newtheorem*{theorem*}{Theorem}
\newtheorem{proposition}[theorem]{Proposition}
\newtheorem{lemma}[theorem]{Lemma}

\newtheorem{corollary}[theorem]{Corollary}
\newtheorem*{corollary*}{Corollary}
\theoremstyle{definition}
\newtheorem{definition}[theorem]{Definition}
\newtheorem{example}[theorem]{Example}
\newtheorem{remark}[theorem]{Remark}

\DeclareMathSymbol\bbDelta\mathord{bbold}{"01}
\newcommand\nc{\newcommand}
\nc\dmo{\DeclareMathOperator}

\dmo\plex{\cal{P}_\mathrm{lex}}
\dmo\Map{\mathrm{Map}}
\dmo\op{\mathrm{op}}
\dmo\ev{ev}
\dmo\Sp{\cal{S}\kern-.5pt p}
\dmo\exact{\cal{E}\kern-.5pt xact}
\dmo\add{\cal{A}dd}
\dmo\wald{\cal{W}ald}
\dmo\cowald{co\cal{W}ald}
\dmo\Ab{\cal{A}b}
\dmo{\zz}{\mathbb{Z}}
\dmo{\colim}{colim}
\dmo{\eeq}{=\joinrel=}
\dmo{\ipe}{\twoheadleftarrow}
\dmo{\epi}{\twoheadrightarrow}
\dmo{\onom}{\leftarrowtail}
\dmo{\mono}{\rightarrowtail}
\dmo\fun{Fun}
\dmo\funex{Fun^{ex}}
\dmo{\cat}{\cal{C}\kern -.5pt at}
\dmo\catex{\cal{C}\kern -.5pt at^{ex}}
\dmo{\id}{id}
\dmo\st{St}
\dmo\pst{pSt}
\dmo\fib{\mathrm{fib}}
\dmo\cofib{\mathrm{cofib}}
\dmo{\pb}{\arrow[dr, phantom, "\scalebox{1.5}{$\lrcorner$}" , very near start]}

\nc{\bb}[1]{\mathbb{#1}}
\nc{\mrm}[1]{\mathrm{#1}}
\nc{\cal}[1]{\EuScript{#1}}
\nc{\xto}[1]{\xrightarrow{#1}}

\nc{\C}{\cal{C}}
\nc{\E}{\cal{E}}
\nc{\Psh}{\cal{P}}
\nc{\ex}{\mathrm{ex}}
\nc{\lex}{\mathrm{lex}}
\nc{\pr}{\mathrm{pr}}
\nc{\vphi}{\varphi}
\nc{\eps}{\epsilon}
\nc{\inc}{\mathrm{in}}
\nc{\Sps}{\cal{S}}
\nc{\reedy}{\mathrm{reedy}}
\nc{\K}{\mathrm{K}}
\nc\ie{i.e.\ }
\nc\hook{hookrightarrow}
\nc\tail{rightarrowtail}
\nc\head{twoheadrightarrow}
\nc{\Span}{\mathrm{Span}}
\nc{\Ar}{\mathrm{Ar}}
\nc{\fw}{\mathrm{fw}}
\nc{\bw}{\mathrm{bw}}
\nc{\rS}{\mathrm{S}}  
\nc{\Grp}{\mathrm{Grp}}
\nc{\Einf}{\mathbb{E}_{\infty}}
\nc{\cocolon}{\nobreak \mskip6mu plus1mu \mathpunct{}\nonscript\mkern-\thinmuskip {:}\mskip2mu \relax}

\nc{\adj}{\mathbin{
\begin{tikzpicture}[baseline,thick] 
\coordinate (source) at (0ex,.5ex);
\coordinate (target) at (3ex,.5ex);
\draw[->] ([yshift=1ex]source) -- ([yshift=1ex]target); 
\draw[->] ([yshift=-.5ex]target) -- ([yshift=-.5ex]source);
\node at (1.5ex,.8ex) {$\scriptscriptstyle \perp$};
\end{tikzpicture}%
}}

\makeatletter

\renewcommand{\tocsection}[3]{%
\indentlabel{\@ifnotempty{#2}{\parbox[b]{3ex}{\bfseries\ignorespaces#1 #2}}}\bfseries#3} 

\renewcommand{\tocsubsection}[3]{%
\indentlabel{\@ifnotempty{#2}{\hspace{1.6em}\parbox[b]{5ex}{\ignorespaces#1 #2}}}#3}

\renewcommand{\tocsubsubsection}[3]{%
\indentlabel{\@ifnotempty{#2}{\hspace{3.9em}\parbox[b]{5ex}{\ignorespaces#1 #2}}}#3}

\makeatother

\DeclareRobustCommand{\SkipTocEntry}[5]{} 

\nc{\introsubsection}[1]{\addtocontents{toc}{\SkipTocEntry}\subsection*{#1}}
\nc{\bibsubsubsection}[1]{\addtocontents{toc}{\SkipTocEntry}\subsubsection*{#1}}

\title{Keller sequences of exact $\infty$-categories}
\author{Yonatan Harpaz}
\author{Daniel Marlowe}

\begin{document}

\begin{abstract}
    We study a family of fibre-cofibre sequences in the $\infty$-category of exact $\infty$-categories, which we call Keller sequences, generalising the notion of Verdier sequences of stable $\infty$-categories. We show that algebraic K-theory is localising with respect to this family, and use this to give a categorical proof of Quillen's resolution theorem. Our arguments hinge upon a mapping space characterisation of special inclusions as studied by Keller, from which we deduce another proof of Keller's theorem that such inclusions induce fully faithful functors on stable envelopes. 
\end{abstract}

\maketitle

\tableofcontents

\section{Introduction}

The existence of a reasonable theory of bifibre sequences of stable $\infty$-categories, known as Verdier sequences, gives rise to a rich theory of localising invariants, as demonstrated in \cite{BGT13, HLS24, RSW25}. A key feature of the theory is that the algebraic $\K$-theory functor 
\[ \K\colon\catex\to\Sp \] 
which can be defined as the initial additive grouplike functor on the $\infty$-category $\catex$ of stable $\infty$-categories under the groupoid core functor $\cal{C} \mapsto \Sigma^{\infty}\cal{C}^\simeq$, is in fact Verdier localising, that is, sends Verdier sequences of stable $\infty$-categories to spaces. Here, the notion of additivity can either be defined using semi-orthogonal decompositions or by preserving \emph{split} Verdier sequences.
It is however mostly non-split Verdier sequences that underline many of the applications of $\K$-theoretic invariants to algebraic geometry and motivic homotopy theory \cite{TT90, CHN24}.

Now the homotopy theory of stable $\infty$-categories sits inside that of exact $\infty$-categories \cite{Bar15} as the full subcategory $\catex \subset \exact$ spanned by exact $\infty$-categories in which every arrow is both an exact inclusion and an exact projection. By work of Klemenc \cite{Kle22}, this inclusion 
admits a left adjoint
\[
    \st\colon\exact \to \catex,
\]
called the stable envelope, which on ordinary exact categories recovers the bounded derived $\infty$-category \cite[Corollary 7.4.12]{BCKW19}, and under which algebraic $\K$-theory is invariant (see~\cite[Theorem 1.7]{SW25}).

This invariance under passage to the stable envelope, which can be viewed as a version of the Gillet-Waldhausen theorem, allows one to bootstrap the universality of $\K$-theory as an additive invariant on $\catex$ to produce a suitable universal property as an invariant on $\exact$ \cite[Theorem 6.7]{SW25}. To do this, one uses that exact $\infty$-categories support a notion of semi-orthogonal decompositions, which under the stable envelope are taken to semi-orthogonal decompositions of stable $\infty$-categories. It is then natural to ask, either for the sake of formal completeness or for K-theoretic consequences, whether the notion of Verdier sequence admits a generalisation to the realm of exact $\infty$-categories.

For this, it is convenient to identify a class of exact subcategory inclusions for which there is some hope of controlling the cofibre. One such candidate was recently studied by Winges \cite{Win25} as a variant of the (left or right) filtering subcategory inclusions of Schlichting \cite{Sch04}, themselves a generalisation of the notion of a Karoubi filtration on an additive category \cite{Kar70, PD89}. Winges proves a general localisation theorem \cite[Theorem 5.2]{Win25} for strongly filtering inclusions amongst exact $\infty$-categories, using the crucial fact that one may model the cofibre of a right filtering inclusion $\cal{A} \subset \cal{C}$ in \emph{Waldhausen} $\infty$-categories as a Dwyer-Kan localisation (and dually for left filtering inclusions). Since Verdier inclusions among stable $\infty$-categories are in particular both left and right filtering, this recovers the localisation theorem \cite[Theorem 6.1]{HLS24} for algebraic K-theory of stable $\infty$-categories, and a host of other localisation results, among which is Quillen's localisation theorem for Serre subcategories of abelian categories \cite{Qui73}. This approach has however an important drawback, in that the cofibre in Waldhausen $\infty$-categories of a right filtering inclusion 
among exact $\infty$-categories is generally not itself exact (even if the inclusion is both left and right filtering, such as the inclusion of finite abelian and finitely generated abelian groups, whose Waldhausen quotient is torsion free abelian groups), and so this approach does not seem to lead to a useful family of bifibre sequences among exact $\infty$-categories.

Somewhat orthogonally to the notion of filtering inclusions lies the class of special inclusions studied by Keller \cite{keller-derived}, and later Schlichting \cite{Sch04}, which again come in left- and right-handed variants: an extension-closed subcategory $i\colon\cal{A}\subset\cal{C}$ is right-special if for each exact inclusion $a \mono x$ in $\cal{C}$ with $a \in \cal{A}$, there exists some map $x \to b$ with $b \in \cal{A}$ such that the composite $a \to b$ is an exact inclusion with cofibre in $\cal{A}$; $i$ is left-special if $i^{\op}$ is right-special. Any full stable subcategory inclusion in $\catex$ is automatically left and right-special, as is any additive inclusion of split-exact $\infty$-categories.

The key observation, due to Keller \cite[Theorem 12.1]{keller-derived}, is that such special inclusions $i\colon\cal{A}\subset\cal{C}$ are compatible with stabilisation in the sense that the induced exact functor $\st(i)\colon\st(\cal{A})\subset\st(\cal{C})$ on stable envelopes is fully faithful. Left adjointness of $\st(-)$ then implies that $\st$ sends any cofibre sequence $\cal{A} \subset \cal{C} \to \cal{C}/\cal{A}$ extending a special inclusion to a Verdier sequence of stable $\infty$-categories, at least up to retract closure.

This is promising for the candidacy of special inclusions as a class of Verdier-like inclusions among exact $\infty$-categories. It may be troubling however to note the asymmetry inherent in the definition of (left or right) speciality, while any such candidate ought to be invariant under the autoequivalence
\[
    (-)^{\op}\colon\exact\simeq\exact,
\]
particularly if one wishes for this notion to be useful in contexts where duality functors are present. Taking this aesthetic preference seriously, we consider in this paper extension-closed full additive subcategories $\cal{A} \subset \cal{C}$ which are both left and right-special, as well as retract-closed, to which we propose the term \emph{Keller inclusions}. Crucially, we show that the cofibre in $\exact$ of a Keller inclusion $\cal{A} \subset \cal{C}$ is modelled as an appropriate Dwyer-Kan localisation (at either the wide subcategory $\cal{C}_\inc^\cal{A}$ of inclusions with cofibre in $\cal{A}$ or at the wide subcategory $\cal{C}_\pr^\cal{A}$ of projections with fibre in $\cal{A}$). This enables us to show that there are corresponding notions of Keller projections, whose relation to that of Keller inclusions as expressed by the following property:
in any nullcomposite sequence
\[
    \cal{A} \xto i \cal{C} \xto p \cal{B}
\]
of exact $\infty$-categories, $i$ is a Keller inclusion and $p$ is its cofibre if and only if $p$ is a Keller projection and $i$ is its fibre, paralleling the behaviour of Verdier sequences \cite[Corollary A.1.10]{9-authors-II}; see Proposition \ref{prop:verdier}. To our knowledge, this gives the first explicit well-behaved family of non-split bifibre sequences in $\exact$ among non-stable, non-split-exact $\infty$-categories. We also show that such Keller sequences are taken to Verdier sequences by the stable envelope (Corollary \ref{st}), which, when combined with the Gillet-Waldhausen theorem for exact $\infty$-categories gives an a-posteriori proof that $\K$-theory is localising with respect to Keller sequences.

However, as it turns out, the Keller localising property of $\K$-theory admits a proof directly from the definitions, using as a technical input that in this setting the canonical functors
\[
    \cal{C}_\inc^\cal{A} \to \cal{B}^\simeq \leftarrow \cal{C}_\pr^\cal{A}
\]
induce equivalences on geometric realisations, which we prove in Proposition \ref{prop:weak-equivalence}. We then obtain a self-contained proof of the following statement:

\begin{theorem*}[Theorem \ref{thm:localisation}]
    Let $\cal{A} \to \cal{C} \to \cal{B}$ be a Keller sequence. Then the sequence
    \[
        \K(\cal{A}) \to \K(\cal{C}) \to \K(\cal{B})
    \]
    is an exact sequence of spectra.
\end{theorem*}

We tie back this circle of ideas by showing that the above localisation theorem, obtained without relying on the Gillet-Waldhausen theorem, can, itself, be used to \emph{prove} the Gillet-Waldhausen theorem by following the approach of \cite[Lecture 19]{Lur14}. More precisely, we still rely on the fact, proven in \cite{SW25}, that stable envelops of exact $\infty$-categories carry heart structures; our alternative argument just handles the passage to the heart (proven in \cite{Sau23} as an $\infty$-categorical analogue of Quillen's resolution theorem). In fact, using this 
approach one can prove the Gillet-Waldhausen theorem for any finitary Keller localising additive invariant on $\exact$:

\begin{corollary*}[Corollary \ref{cor:resolution}]
	Let $\cal{F}\colon \exact \to \cal{A}$ be a filtered colimit preserving Keller localising additive functor valued in an additive presentable $\infty$-category $\cal{A}$. Then the map $\cal{F}(\cal{C}) \to \cal{F}(\st(\cal{C}))$ is an equivalence.
\end{corollary*}

Combining this with the universal property of algebraic $\K$-theory on stable $\infty$-categories one immediately obtains:

\begin{corollary*}
The algebraic $\K$-theory functor $\K\colon \exact \to \Sp$ is the initial Keller localising functor under $\Sigma^{\infty}(-)^{\simeq}$.
\end{corollary*}

\subsection*{Organisation}

We begin by establishing some preliminaries on exact $\infty$-categories, stable envelopes and heart structures in \S\ref{sec:preliminaries}. In \S\ref{sec:keller-sequences} we introduce and study the main concept of this paper, namely that of Keller sequences. We first recall in \S\ref{sec:keller-inclusions} Keller's notion of special inclusions and establish some of their basic properties. In \S\ref{sec:keller's-theorem} we prove a mapping space characterisation (Corollary~\ref{fact}) of specialty which yields another proof of Keller's theorem (Theorem \ref{thm:fully-faithful}), and provides the key technical foundation for the following sections. In \S\ref{sec:calculus} we show that localisations by left- and right-special subcategories admit respectively a right and left calculus of fractions (Lemma \ref{lem:calculus}), which we use in \S\ref{subsec:keller-sequences} to prove a characterisation of Keller sequences (Proposition \ref{prop:verdier}). A useful test case for the behaviour of Keller sequence is presented by Reedy diagram categories, an example we analyse thoroughly in \S\ref{sec:reedy}. 

The final \S\ref{sec:K-theory} is dedicated to algebraic K-theory of exact $\infty$-categories.
After discussing additivity of $\K$-theory in \S\ref{sec:additivity}, we prove the main $\K$-theoretic result, namely that algebraic $\K$-theory is localising with respect to Keller sequences (Theorem \ref{thm:localisation}) in section \ref{sec:k-theory}. In section \ref{sec:resolution} we use this to give a new proof of the $\infty$-categorical resolution, and consequently Gillet-Waldhausen, theorem (Corollary~\ref{cor:resolution}).

\subsection*{Acknowledgements}

The first author was supported by the European Research Council as part of the project ``Foundations of Motivic Real K-Theory'' (ERC grant no. 949583).
We would like to thank Christoph Winges for valuable comments on a draft of this article, and for kindly allowing us to reproduce his argument for colimits of exact $\infty$-categories in \S\ref{sec:preliminaries}. The beginning of this investigation was conducted during a visit of the second author to Universit\'e Paris Cit\'e, for which he gratefully acknowledges funding from the University of Warwick, and the hospitality of the IMJ-PRG.

\subsection*{LLM declaration} 

No Artificial Intelligence tool was used in the mathematical research or writing of the present paper. 

\section{Preliminaries}\label{sec:preliminaries}

Recall from~\cite{Sau23} that a heart structure on a stable $\infty$-category $\cal{C}$ consists of a pair of full subcategories $\cal{C}_{\geq 0},\cal{C}_{\leq 0} \subseteq \cal{C}$ satisfying the following properties:
\begin{enumerate}
\item
$\cal{C}_{\geq 0}$ is closed under finite colimits, $\cal{C}_{\leq 0}$ is closed under finite limits, and both are closed under extensions.
\item
Every object $x \in \cal{C}$ sits in an exact sequence $z \to x \to y$ with $z \in \cal{C}_{\leq 0}$ and $y \in \Sigma \cal{C}_{\geq 0}$.
\end{enumerate}
The $\infty$-category $\cal{C}^{\heartsuit} := \cal{C}_{\geq 0} \cap \cal{C}_{\leq 0}$ is called the heart of the heart structure. We refer to the tuple $(\cal{C},\cal{C}_{\geq 0},\cal{C}_{\leq 0})$ as a heart $\infty$-category.
Given a heart structure one writes $\cal{C}_{\geq n} := \Sigma^n\cal{C}, \cal{C}_{\leq n} := \Sigma^n\cal{C}_{\leq n}$ and $\cal{C}_{[m,n]} = \cal{C}_{\geq m} \cap \cal{C}_{\leq n}$. For the sake of uniformity of notation it is also convenient to write $\cal{C}_{[n,\infty]} := \cal{C}_{\geq n}$ and $\cal{C}_{[-\infty,n]} := \cal{C}_{\leq n}$, so that the notation $\cal{C}_{[m,n]}$ makes sense for every $-\infty \leq m \leq n \leq \infty$.

We say that a heart structure is bounded if every $x \in \cal{C}$ belongs to $\cal{C}_{[m,n]}$ for some $-\infty < m \leq n < \infty$.

\begin{lemma}\label{lem:closed-under-retracts}
Let $\cal{C}$ be a stable $\infty$-category equipped with an bounded heart structure
$(\cal{C}_{\geq 0},\cal{C}_{\leq 0})$. Then for every $-\infty \leq m \leq n \leq \infty$ the $\infty$-category $\cal{C}_{[m,n]}$ is closed under retracts in $\cal{C}$. In particular, it is weakly idempotent complete.
\end{lemma}
\begin{proof}
We first show that if $n < -\infty$ then $\cal{C}_{[m,n]}$ is closed under retracts in $\cal{C}_{[m,n+1]}$. For this, let $x \to y \to x$ be a retract diagram in $\cal{C}_{[m,n+1]}$ with $y \in \cal{C}_{[m,n]}$ and set $z = \fib[y \to x] =\cofib[x \to y] \in \cal{C}$. Now the exact sequence $x \to y \to z$ implies that $z \in \cal{C}_{\geq m}$ while the exact sequence $\Omega x \to z \to y$ implies that $z \in \cal{C}_{\leq n}$. We conclude that $z \in \cal{C}_{[m,n]}$. The exact sequence $x \to y \to z$ now implies that $x \in \cal{C}_{[m,n]}$, as desired.

Replacing $(\cal{C},\cal{C}_{\geq 0},\cal{C}_{\leq 0})$ with the opposite heart $\infty$-category $(\cal{C}^{\op},(\cal{C}_{\leq 0})^{\op},(\cal{C}_{\geq 0})^{\op})$ we conclude that if $-\infty < m$ then $\cal{C}_{[m,n]}$ is closed under retracts in $\cal{C}_{[m-1,n]}$. 
By induction we conclude that for any $-\infty < m'\leq m \leq n \leq n' < \infty$ the inclusion $\cal{C}_{[m,n]} \subseteq \cal{C}_{[m',n']}$ is closed under retracts. Since the heart structure is assumed bounded it now follows that $\cal{C}_{[m,n]}$ is closed under retracts in $\cal{C}$ for every $-\infty \leq m \leq n \leq \infty$. 
\end{proof}

Let $\exact$ denote the $\infty$-category of small exact $\infty$-categories and exact functors between them. We note that exact $\infty$-categories in which all maps are both inclusions and projections are exactly the stable $\infty$-categories, and the exact $\infty$-categories in which every map is an inclusion are exactly the prestable $\infty$-categories. In addition, in the latter case exact functors correspond exactly the functors preserving finite colimits. The resulting full inclusions of stable and prestable $\infty$-categories in $\exact$ both admit left adjoints, given by the formation of stable and prestable envelopes, which we denote by $\cal{E} \mapsto \st(\cal{E})$ and $\cal{E} \mapsto \pst(\cal{E})$. The exact functors
\[ \cal{E} \to \pst(\cal{E}) \to \st(\cal{E}) \]
are both extension-closed full inclusions which detect exact sequences, see \cite[Proposition 4.25]{Kle22}. In addition, the functor $\pst(\cal{E}) \to \st(\cal{E})$ exhibits its target both as the stable envelope of its source, considered as an exact $\infty$-category, and as the Spanier-Whithead stabilization of its source, considered as a prestable $\infty$-category. Finally, by \cite[Proposition 5.4]{SW25} the stable $\infty$-category $\st(\cal{E})$ carries a canonical bounded heart structure satisfying
\begin{enumerate}
\item
$\st(\cal{E})_{\geq 0} = \pst(\cal{E})$.
\item
$\cal{E} \subseteq \st(\cal{E})^{\heartsuit}$ and exhibits $\st(\cal{E})^{\heartsuit}$ as the weak idempotent completion of $\cal{E}$.
\end{enumerate}
In particular, if $\cal{E}$ is weakly idempotent complete then $\st(\cal{E})^{\heartsuit} = \cal{E}$. Combining this with Lemma~\ref{lem:closed-under-retracts} we conclude:

\begin{corollary}\label{cor:retract-closure}
Let $\cal{E}$ be a weakly idempotent complete exact $\infty$-category. Then $\cal{E}$ is closed under retracts in $\st(\cal{E})$. 
\end{corollary}

In the present paper we will frequently consider fibres and cofibres in the category of small exact $\infty$-categories. Though our arguments show the requisite (co)limits exist by constructing them, it is useful to know that $\exact$ admits in general all small (co)limits. While the argument is quite standard for limits, the one for colimits is slightly more involved, and we thank Christop Winges for sharing it with us.

\begin{proposition}
    The $\infty$-category $\exact$ admits all small limits.
\end{proposition}
\begin{proof}
Given a small $\infty$-category $I$ and a diagram 
\[ \cal{C}_\bullet\colon I \to \exact \quad\quad i \mapsto \cal{C}_i ,\] 
let us write $\cal{D} := \lim_IU(\cal{C}_\bullet)$ for its limit in $\cat_\infty$, where $U \colon\exact \to\cat_\infty$ is the forgetful functor. 
It follows from \cite[Proposition 5.1.2.2]{HTT} that a diagram in $\cal{D}$ admits a (co)limit as soon as its image under each of the structure maps $p_i\colon\cal{D} \to U(\cal{C}_i)$ admits a (co)limit which is furthermore preserved under each of the transition functors in $U(\cal{C}_i) \to U(\cal{C}_j)$ in the diagram, in which case the corresponding (co)limit in $\cal{D}$ is computed pointwise. Since all the $U(\cal{C}_i)$ are additive and all the transition functors $U(\cal{C}_i) \to U(\cal{C}_j)$ are direct sum preserving it now follows that $\cal{D}$ is additive and each of the structure maps $p_i$ is direct sum preserving. By a similar argument, if we declare that a map in $\cal{D}$ is an inclusion/projection if and only if its image under all the $p_i$ is an inclusion/projection then the resulting classes of maps form an exact structure on $\cal{D}$. A functor $f\colon\cal{E} \to \cal{D}$ from an exact category $\cal{E}$ is then exact with respect to this structure precisely if each of the composites $p_i\circ f\colon\cal{E} \to \cal{C}_i$ is so, and so this exact structure exhibits $\cal{D}$ as the limit of this diagram in $\exact$.

\end{proof}

For colimits, the following argument is due to Christoph Winges. Let $\wald$ be the $\infty$-category of (small) Waldhausen $\infty$-categories and $\C \in \wald$ be a Waldhausen $\infty$-category.
The $\infty$-category of left-exact presheaves $\Psh_\lex(\C)$ is defined to be the full subcategory of $\Psh(\C)$ spanned by those space-valued presheaves $X \colon \C^{\op} \to \cal{S}$ which send the zero object of $\C$ to $\ast \in \cal{S}$ and pushout squares in $\cal{C}$ with parallel leg inclusions to cartesian squares in $\cal{S}$. Then $\Psh_\lex(\cal{C})$ is presentable and for every presentable $\infty$-category $\cal{E}$, restriction along the Yoneda embedding
\[ i_\cal{C} \colon \cal{C} \to \Psh_\lex(\C) \]
induces an equivalence between colimit preserving functors $\Psh_\lex(\C) \to \cal{E}$ and Waldhausen functors $\C \to \E$ with respect to the maximal Waldhausen structure on $\E$. 
We then define $\Psh_\lex^\mrm{st}(\C) := \Sp(\Psh_\lex(\C))$ to be the associated stabilisation. In particular, $\Psh_\lex^\mrm{st}(\C)$ is stable and enjoys the same type of universal mapping property as $\Psh_\lex(\cal{C})$ for targets which are stable presentable.

\begin{definition}
 Let $\C$ be a small Waldhausen $\infty$-category.
 Define the small exact $\infty$-category $\C^\ex$ as the extension-closed subcategory of $\Psh_\lex^\mrm{st}(\C)$ generated by the essential image of
 \[ i^{\mrm{st}}_\C \colon \C \to \Psh_\lex(\C) \to \Psh_\lex^\mrm{st}(\C). \]
 By definition, $\C^\ex$ comes equipped with an exact functor $u \colon \C \to \C^\ex$.
\end{definition}

\begin{proposition}
    Let $\C$ be a small Waldhausen $\infty$-category and let $\E$ be an exact $\infty$-category.
    Then the restriction functor
    \[ u^* \colon \fun^\ex(\C^\ex,\E) \to \fun^\ex(\C,\E) \]
    is an equivalence. 
    Consequently, the association $\C \mapsto \C^{\ex}$ assembles to form a left adjoint to the 			fully-faithful inclusion $\exact \hookrightarrow \wald$.
\end{proposition}
\begin{proof}
    From the universal property of $\Psh_\lex^\mrm{st}(\C^\ex)$, we obtain a commutative square
    \[\begin{tikzcd}
        \C\ar[r, "u"]\ar[d, "i^{\mrm{st}}_\C"'] & \C^\ex\ar[d, "i^{\mrm{st}}_{\C^\ex}"] \\
        \Psh_\lex^\mrm{st}(\C)\ar[r, "U"] & \Psh_\lex^\mrm{st}(\C^\ex)
    \end{tikzcd}\]
    where $U$ is a colimit-preserving functor.
    Moreover, the exact inclusion functor $j \colon \C^\ex \to \Psh_\lex^\mrm{st}(\C)$ induces a colimit-preserving functor
    \[ J \colon \Psh_\lex^\mrm{st}(\C^\ex) \to \Psh_\lex^\mrm{st}(\C). \]
    Since $JUi^{\mrm{st}}_\C \simeq Ji^{\mrm{st}}_{\C^\ex}u \simeq ju \simeq i^{\mrm{st}}_\C$, the composite $JU$ is the identity functor.

    We claim that $UJ$ is also the identity.
    To see this, note that $i^{\mrm{st}}_{\C^\ex} \colon \C^\ex \to \Psh_\lex^\mrm{st}(\C^\ex)$ exhibits $\C^\ex$ as an extension-closed subcategory.
    Consequently, its preimage $U^{-1}(\C^\ex)$ is an extension-closed subcategory of $\Psh_\lex^\mrm{st}(\C)$.
    Since $\C$ is contained in $U^{-1}(\C^\ex)$, it follows that $\C^\ex \subseteq U^{-1}(\C^\ex)$, and it follows that $UJ$ restricts to an exact endofunctor $f \colon \C^\ex \to \C^\ex$.
    From $UJi^{\mrm{st}}_{\C^\ex} \simeq i^{\mrm{st}}_{\C^\ex}f$ we conclude that
    \[ j \simeq Ji^{\mrm{st}}_{\C^\ex} \simeq JUJi^{\mrm{st}}_{\C^\ex} \simeq Ji^{\mrm{st}}_{\C^\ex}f \simeq jf. \]
    Since $j$ is fully faithful, it follows that $f$, and thus also $UJ$, is the identity functor. In particular, $J$ and $U$ are inverse equivalences.
    Now from the commutative diagram
    \[\begin{tikzcd}
     \fun^\ex(\C^\ex,\E)\ar[r]\ar[d, , "u^*"'] & \fun^\ex(\C^\ex,\Psh_\lex^\mrm{st}(\E))\ar[d] & \fun^\ex(\Psh_\lex^\mrm{st}(\C^\ex), \Psh_\lex^\mrm{st}(\E))\ar[l ,"\sim"']\ar[d, "U^*", "\sim"'] \\
     \fun^\ex(\C,\E)\ar[r] & \fun^\ex(\C,\Psh_\lex^\mrm{st}(\E)) & \fun^\ex(\Psh_\lex^\mrm{st}(\C), \Psh_\lex^\mrm{st}(\E))\ar[l, "\sim"']
    \end{tikzcd}\]
    we conclude that that the middle vertical arrow is an equivalence.
    Since $\E \to \Psh_\lex^\mrm{st}(\E)$ exhibits $\E$ as an extension-closed subcategory, the left square is a pullback.
    Consequently, $u^*$ is an equivalence.
\end{proof}

Now recall from \cite[Proposition 4.13]{Bar16} that the $\infty$-category $\wald$ is compactly generated, and in particular admits small colimits. Having established that $\exact$ is a reflective subcategory of $\wald$ we conclude:

\begin{corollary}
The $\infty$-category $\exact$ admits small colimits.
\end{corollary}

Furthermore, we claim that:
\begin{corollary}\label{cor:compactly-generated}
The reflective inclusion $\exact \subseteq \wald$ is closed under filtered colimits.
As a consequence, the $\infty$-category $\exact$ is compactly generated and the forgetful functor $\exact \to \cat$ preserves filtered colimits.
\end{corollary}
\begin{proof}
Let $I$ be a filtered $\infty$-category and $\cal{C}_\bullet\colon I \to \exact$ a diagram. We wish to show that the colimit $\cal{D}$ of $\{\cal{C}_i\}_{i \in I}$ in $\wald$ is exact. For this, note that filtered colimits in $\wald$ are computed as follows: the underlying $\infty$-category of $\cal{D}$ is the colimit $\colim_i U(\cal{C}_i) \in \cat$ of the underlying $\infty$-categories $U(\cal{C}_i)$. In particular, for each $i$ the resulting functor $U(\cal{C}_i) \to \cal{D}$ preserves all finite (co)limits which exist in $U(\cal{C}_i)$ and which are preserved by the transition functors $U(\cal{C}_i) \to U(\cal{C}_j)$ for all $[i \to j] \in I_{i/}$ (where we note that $I_{i/}$ is filtered and $I_{i/} \to I$ is cofinal, so that we can assume if we wish that the colimit is indexed by $I_{i/}$). In addition, every finite diagram in $\cal{D}$ factors through one of the $U(\cal{C}_i) \to \cal{D}$ by compactness. Consequently, since each $U(\cal{C}_i)$ is additive and each transition functor $U(\cal{C}_i) \to U(\cal{C}_j)$ in the diagram is direct sum preserving we have that $\cal{D}$ is additive. The colimit Waldhausen structure on $\cal{D}$ is then such that a map is an inclusion if and only if it is the image of an inclusion in $\cal{C}_i$ for some $i \in I$. Then every span $x \leftarrow y \mono z$ in $\cal{D}$ with right leg inclusion comes from a span with right leg inclusion in some $\cal{C}_i$, and so we see that this is indeed a Waldhausen structure and that the universal property of colimit (in $\wald$) holds.

Since we already know that $\cal{D}$ is additive, to show that $\cal{D}$ is exact it will suffice to show that there is some exact structure on $\cal{D}$ with this collection of inclusions. Indeed, define a map to be a projection if it is the image of a projection in $\cal{C}_i$ for some $i \in I$. 
Since any of the following types of diagrams
\[
\begin{tikzcd}
x \ar[r,hookrightarrow]\ar[d] & y \\
z &
\end{tikzcd}
\quad\quad
\begin{tikzcd}
x \ar[r,hookrightarrow]\ar[d,->>] & y \\
z &
\end{tikzcd}
\quad\quad
\begin{tikzcd}
& y\ar[d,->>] \\
z\ar[r] & w
\end{tikzcd}
\quad\quad
\begin{tikzcd}
& y\ar[d,->>] \\
z\ar[r,hookrightarrow] & w
\end{tikzcd}
\]
in $\cal{D}$ comes a diagram of the same type in one of the $\cal{C}_i$ we see that this is indeed an exact structure, as desired.
\end{proof}

\section{Keller sequences}\label{sec:keller-sequences}

\subsection{Keller inclusions}\label{sec:keller-inclusions}

In what follows, we will say that an exact functor $i\colon \cal{A} \to \cal{C}$ of exact $\infty$-categories is an \emph{extension-closed full inclusion} if it satisfies the following conditions:
\begin{itemize}
\item
The functor underlying $i$ is fully-faithful.
\item
The essential image of $i$ is closed under extensions.
\item
The functor $i$ detects exact sequences.
\end{itemize}
In particular, if $\cal{C}$ is an exact $\infty$-category then the data of an extension-closed full inclusion $i\colon \cal{E} \to \cal{C}$ simply corresponds to a full subcategory $\cal{A} \subseteq \cal{C}$ closed under extensions, equipped with the induced exact structure in which a map in $\cal{A}$ is an inclusion exactly if its an inclusion in $\cal{C}$ with cofibre in $\cal{A}$ and a projection exactly if its a projection in $\cal{C}$ with cofibre in $\cal{A}$.

\begin{definition}\label{keller-inclusion}
	An extension-closed full inclusion $i\colon \cal{A} \hookrightarrow \cal{C}$ of exact categories is \emph{left-special} if for every projection $p\colon y \epi i(a)$ in $\cal{C}$ with $a \in \cal{A}$ there is a projection $q\colon b \epi a$ in $\cal{A}$ and a factorization $i(b) \to y \to i(a)$ of $i(q)$ through $p$.
Dually, $i\colon \cal{A} \hookrightarrow \cal{C}$ is \emph{right-special} if $i^{\op}$ 
is left-special. We say that $i$ is a \emph{Keller inclusion} if it is closed under retracts and both left and right-special. 
\end{definition}

\begin{example}\label{ex:verdier-inclusion}
If $\cal{A}$ and $\cal{C}$ are both stable $\infty$-categories with the associated maximal exact structure then every fully faithful inclusion $\cal{A} \to \cal{C}$ is both left and right-special. Similarly, if $\cal{A}$ and $\cal{C}$ are both prestable $\infty$-categories with the associated exact structures in which all maps are inclusions then every fully faithful inclusion $\cal{A} \to \cal{C}$ is right-special.
\end{example}

\begin{example}\label{ex:adjoints}
If $i \colon \cal{A} \to \cal{C}$ admits a right adjoint $r\colon \cal{C} \to \cal{A}$ which preserves projections then $i$ is left-special: indeed, given a projection $y \epi i(a)$ we can take $b = r(y) \epi r(i(a)) \simeq a$, in which case the map $i(b) \to i(a)$ factors as $i(b) = ir(y) \to y \epi i(a)$ where the first map is the counit.
Similarly, if $i$ admits an inclusion preserving left adjoint then $i$ is right-special. 
\end{example}

\begin{example}\label{ex:arrow}
Let $\cal{E}$ be an exact $\infty$-category and $\mathrm{Ar}(\cal{E}) = \fun(\Delta^1,\cal{E})$ the arrow category of $\cal{E}$ with the pointwise exact structure. Then the extension-closed exact inclusion $i\colon \cal{E} \to \mathrm{Ar}(\cal{E})$ sending $x$ to $\id\colon x \to x$ admits exact left and right adjoints given by the target and source projections, respectively, and so $i$ is both left and right-special by Example~\ref{ex:adjoints}. On the other hand, the extension-closed inclusion $x \mapsto [x \to 0]$ is right-special but not left-special, and the extension-closed inclusion $y \mapsto [0 \to y]$ is left-special but not right-special.
\end{example}

\begin{example}\label{ex:resolving-is-special}
If $i\colon \cal{A} \to \cal{C}$ detects projections and every object in $\cal{C}$ admits a projection from an object of $\cal{A}$ then $i$ is left-special. Similarly, if $i$ detects inclusions and every object of $\cal{C}$ admits an inclusion into an object of $\cal{A}$ then $i$ is right-special. For example, for any exact $\infty$-category $\cal{E}$ the inclusions $\cal{E} \to \st(\cal{E})_{[0,n]}$ and $\cal{E} \to \st(\cal{E})_{[m,0]}$ are left-special and right-special, respectively, for every $-\infty \leq m \leq 0$ and $0 \leq n \leq \infty$.
\end{example}

In a similar spirit, we also have the following:
\begin{lemma}\label{lem:square-left-special}
Let
\[ 
\begin{tikzcd}
\cal{A} \ar[r,hookrightarrow]\ar[d,hookrightarrow] & \cal{C} \ar[d,hookrightarrow] \\
\cal{A}'\ar[r,hookrightarrow] & \cal{C}'
\end{tikzcd}
\]
be a commutative square of exact $\infty$-categories in which all arrows are extension-closed inclusions and the vertical arrows detect projections. Assume that every object in $\cal{A}'$ admits a projection from an object in the image of $\cal{A}$ and every object in $\cal{C}'$ admits a projection from an object in the image of $\cal{C}$. If the top horizontal map is left-special then so is the bottom horizontal map.
\end{lemma}
\begin{corollary}
If $\cal{A} \hookrightarrow \cal{C}$ is left-special then $\st(\cal{A})_{[0,n]} \to \st(\cal{C})_{[0,n]}$ is left-special for every $0\le n \leq \infty$ and if $\cal{A} \hookrightarrow \cal{C}$ is right-special then $\st(\cal{A})_{[-m,0]} \to \st(\cal{C})_{[-m,0]}$ is right-special for every $-\infty \leq m \leq 0$.
\end{corollary}

\begin{proof}[Proof of Lemma~\ref{lem:square-left-special}]
Let $c' \epi a'$ be a projection in $\cal{C}'$ with $a'\in \cal{A}'$. By assumption there exists a projection $a \to a'$ in $\cal{A}'$ with $a \in \cal{A}$. For the pullback square
\[
\begin{tikzcd}
d \ar[r,->>]\ar[d,->>] & a \ar[d,->>] \\
c' \ar[r,->>] & a'
\end{tikzcd}
\]
in $\cal{C'}$. By assumption there exists a projection $c \epi d$ with $c \in \cal{C}$. The composite $c \epi d \epi a$ is then a projection in $\cal{C}'$ among objects in $\cal{C}$ and hence a projection in $\cal{C}$ by assumption. Since $\cal{A} \hookrightarrow \cal{C}$ is left-special there exists a projection $b \epi a$ in $\cal{A}$ whose image in $\cal{C}$ factors through $c$. The composite $b \epi a \epi a'$ is then a projection in $\cal{A}'$ whose image in $\cal{C'}$ factors through $c'$, and so the bottom horizontal arrow is left-special.
\end{proof}

\begin{lemma}\label{spec}
	Let $i\colon \cal{A} \hookrightarrow \cal{C}$ an extension-closed full inclusion of exact categories. Then the following are equivalent:
	\begin{enumerate}[label=(\roman*)]
		\item\label{i} $i$ is left-special.
		\item\label{ii} Every map $f\colon \Sigma^{-1}i(a) \to x$ in $\st(\cal{C})$, for $x \in \cal{C}$ and $a \in \cal{A}$, factors as $\Sigma^{-1}i(a) \to i(b) \to x$ for some map $g\colon \Sigma^{-1}a \to b$ in $\st(\cal{A})$ with $b \in \cal{A}$. 
	\end{enumerate}
\end{lemma}
\begin{proof}
	\ref{i} $\Leftrightarrow$ \ref{ii}: Since $\cal{C}$ is closed under extensions in $\st(\cal{C})$ the data of a map $\Sigma^{-1}i(a) \to x$ in $\st(\cal{C})$ with $a \in \cal{A},c \in \cal{C}$ is equivalent to that of an exact sequence $x \mono y \epi i(a)$ in $\cal{C}$ with $a \in \cal{A}$, which in turn is equivalent to the data of a projection $y \epi i(a)$ in $\cal{C}$ with $a \in \cal{A}$. Then in the diagram
	\[\begin{tikzcd}
		\Sigma^{-1}i(a) \ar[r] \ar[d, equal] & x \ar[r, rightarrowtail] & y \ar[r, twoheadrightarrow] & i(a) \ar[d, equal] \\
		\Sigma^{-1}i(a) \ar[r,dotted] & i(b) \ar[u, dotted] \ar[r, rightarrowtail] & i(b') \ar[u, dashed] \ar[r, twoheadrightarrow,dashed] & i(a)
	\end{tikzcd}\]
	the existence of the dashed and dotted extensions with $b,b'\in \cal{A}$ is equivalent, as one can pass between them by taking $b = \fib[b'\epi a]$ and $b'=\cofib[\Sigma^{-1}a \to b]$. 
\end{proof}

\begin{remark}
	Dually, $i\colon \cal{A} \hookrightarrow \cal{C}$ is right-special if and only if any map $x \to \Sigma i(a)$ in $\st(\cal{C})$ with $x \in \cal{C}$ and $a \in \cal{A}$ factors as $x \to i(b) \to \Sigma i(a)$ for some map $b \to \Sigma a$ in $\st(\cal{A})$ with $b \in \cal{A}$.  
\end{remark}

\subsection{Keller's theorem}\label{sec:keller's-theorem}

Recall Keller's theorem~\cite[Theorem 12.1]{keller-derived} (proven in the context of exact $\infty$-categories in~\cite[Theorem 1.2]{SW25}):

\begin{theorem}[Keller]\label{thm:fully-faithful}
Let $i\colon \cal{A} \hookrightarrow \cal{C}$ be an either right-special or left-special inclusion. Then the induced functor $\st(\cal{A}) \to \st(\cal{C})$ is fully faithful.
\end{theorem}

In this section we give another proof of Keller's theorem. Along the way, we develop some of the machinery we will use later on.

\begin{lemma}\label{lem:factor}
Let $i\colon \cal{A} \hookrightarrow \cal{C}$ be an extension-closed exact functor among weakly idempotent complete exact $\infty$-categories.
\begin{enumerate}
\item
If $i$ is right-special then for every $n \geq 0$ and every map $c \to i(A)$ with $c \in \cal{C}$ and $A \in \st(\cal{A})_{[0,n]}$ factors as a composite $c \to i(a) \xrightarrow{i(g)} i(A)$ for some $g\colon a \to A$ with $a \in \cal{A}$ and $\cofib[a \to A] \in \st(\cal{A})_{[1,n]}$.
\item
If $i$ is left-special then for every $m \leq 0$ and every map $i(A) \to c$ with $c \in \cal{C}$ and $A \in \st(\cal{A})_{[m,0]}$ factors as a composite $i(A) \xrightarrow{i(g)} i(a)$ for some $g\colon A \to a$ with $a \in \cal{A}$ and $\fib[A \to a] \in \st(\cal{A})_{[m,-1]}$.
\end{enumerate}
\end{lemma}
\begin{proof}
We prove the first claim, the second then follows by passing to opposite exact $\infty$-categories.
Choose a heart decomposition
	\[
		A' \to A \to \Sigma^n a
	\] 
	with $A' \in \st(\cal{A})_{[0,n-1]}$ and $a \in \cal{A}$. Since the inclusion $\cal{C} \subseteq \pst(\cal{C})$ is left-special by Example~\ref{ex:resolving-is-special}, Lemma~\ref{spec} inductively yields a factorization of the composite $c \to \Sigma^ni(a)$ as 
	\[ c \to \Sigma c_1 \to ... \to \Sigma^{n-1} c_{n-1} \to \Sigma^ni(a) \]
	for some $c_1,...,c_{n-1} \in \cal{C}$.
Now since $\cal{A} \subseteq \cal{C}$ is right-special, the last of these maps moreover factors by (the dual of) Lemma~\ref{spec}\ref{ii} as $\Sigma^{n-1}c' \to \Sigma^{n-1}i(b) \to \Sigma^ni(a)$ for some $b\in\cal{A}$ and map $\Sigma^{n-1}b \to \Sigma a$ in $\pst(\cal{A})$, so that $\cofib(\Sigma^{n-1}b \to \Sigma^na) \simeq \Sigma^na'$ for some $a' \in \cal{A}$. The pullback $B := \Sigma^{n-1}b \times_{\Sigma^na}A$ then participates in an exact sequence $A' \to B \to \Sigma^{n-1}b$ and hence lies in $\st(\cal{A})_{[0, n-1]}$, and moreover there is an equivalence $\cofib(B \to A) \simeq \cofib(\Sigma^{n-1}b \to \Sigma^na) \in \Sigma^n\cal{A} \subseteq \st(\cal{A})_{[1,n]}$. By the universal property of pullbacks the map $\phi\colon c \to i(A)$ factors as $c \to i(B) \to i(A)$.
\end{proof}

\begin{lemma}\label{catoffact}
	Let $i\colon \cal{A} \hookrightarrow \cal{C}$ be a right-special inclusion of exact categories. Given a map $\phi\colon c \to i(A)$ in $\pst(\cal{C})$ with $c \in \cal{C}$ and $A \in \pst(\cal{A})$, the $\infty$-category 
	\[ \cal{A}_{c//A} := \cal{A}_{/A} \times_{\pst(\cal{C})_{/i(A)}}(\pst(\cal{C})_{/A})_{\phi/}\] 
of factorisations of $\phi$ through an object of $\cal{A}$, is weakly contractible.
\end{lemma}
\begin{proof}
Choose $n\ge 0$ such that $A$ lies in $\st(\cal{A})_{[0,n]}\subseteq\st(\cal{C})_{[0,n]}$, which exists since the heart structure on $\st(\cal{A})$ is bounded (and the functor $\st(\cal{A})\subseteq\st(\cal{C})$ is heart exact). For every $m \in \{0,...,n\}$ consider the factorization $\infty$-category
\[ (\st(\cal{A})_{[0,m]})_{c//A} := (\st(\cal{A})_{[0,m]})_{/A} \times_{\st(\cal{C})_{/i(A)}}(\st(\cal{C})_{/A})_{\phi/}. \] 
We claim that the inclusion
\[ (\st(\cal{A})_{[0,m]})_{c//A} \to (\st(\cal{A})_{[0,m+1]})_{c//A}\]
is cofinal for every $m=0,...,n-1$, and in particular a weak homotopy equivalence. Since $(\st(\cal{A})_{[0,n]})_{c//A}$ admits a terminal object this will finish the proof. Next we note that the comma $\infty$-categories associated to these inclusions are of the form $(\st(\cal{A})_{[0,m]})_{c//A'}$ for some $A'\in \cal{A}_{[0,m+1]}$. Replacing $m$ with $n$ it will hence suffice to show that $(\st(\cal{A})_{[0,n-1]})_{c//A}$ is weakly contractible.

Let $(\st(\cal{A})_{[0,n-1]})^{\circ}_{c//A} \subseteq (\st(\cal{A})_{[0,n-1]})_{c//A}$ be the full subcategory spanned by those factorizations $c \to i(B) \to i(A)$ such that $\cofib[B \to A] \in \st(\cal{A})_{[1,n]}$. Then $(\st(\cal{A})_{[0,n-1]})^{\circ}_{c//A}$ is non-empty by Lemma~\ref{lem:factor}. We now claim that $(\st(\cal{A})^{\circ}_{[0,n-1]})_{c//A}$ admits binary cartesian products, and is hence co-sifted, and in particular weakly contractible. Indeed, since the functor $\Map_{\pst(\cal{C})}(c,i(-))$ is product preserving it will suffice by~\cite[Lemma 6.3.3.6]{HTT} to show that the full subcategory $(\pst(\cal{A})^{\circ}_{[0,n-1]})_{/A} \subseteq  (\st(\cal{A})^{\circ}_{[0,n]})_{/A}$ spanned by those $B \to A$ whose cofibre lies in $\st(\cal{A})_{[1,n]}$ has binary products. Indeed, this is a full subcategory of the over category $\st(\cal{A})_{/A}$ spanned by the maps $B \to A$ such that $B$ and $\fib[B \to A]$ belong to $\st(\cal{A})_{[0,n]}$; this over category then has finite products and the full subcategory in question is closed under binary products since $\st(\cal{A})_{[0,n-1]}$ is closed under extensions in $\st(\cal{A})$. 

It is left to show that $(\st(\cal{A})_{[0,n-1]})^{\circ}_{c//A} \subseteq (\st(\cal{A})_{[0,n-1]})_{c//A}$ is cofinal. Indeed, inspecting the associated comma $\infty$-categories we see that they are of the same form $(\st(\cal{A})_{[0,n-1]})^{\circ}_{A'//A}$ (defined this time with respect to the identity $\cal{A} \to \cal{A}$ in place of $i$) with $A' \in \st(\cal{A})_{[0,n-1]}$. These are then again weakly contractible by the above argument, and so the proof is complete.
\end{proof}

\begin{corollary}\label{cor:detects}
Let $i\colon \cal{A} \hookrightarrow \cal{C}$ be a retract-closed right-special inclusion. If a map $f\colon a \to b$ in $\cal{A}$ is such that $i(f)$ is an inclusion in $\cal{C}$, then $f$ is an inclusion in $\cal{A}$. Dually, any retract-closed left-special inclusion detects projections.
\end{corollary}
\begin{proof}
Let $A := \cofib(f)$. Then $c = i(A) = \cofib(i(f)) \in \pst(\cal{C})$ lies in $\cal{C}$. By Lemma~\ref{catoffact} the identity on $c$ factors through $i(a)$ for some $a \in \cal{A}$. Since $i$ is retract closed we get that $\cofib(i(f)) \in \cal{A}$, and hence $f$ is an inclusion in $\cal{A}$. The dual statement follows by replacing $i$ with $i^{\op}$.
\end{proof}

\begin{corollary}\label{fact}
	Given a right-special inclusion $i\colon \cal{A} \hookrightarrow \cal{C}$ and a map $c \to C$ in $\pst(\cal{C})$ with $c \in \cal{C}$, the functor 
	\[ \Map_{\pst(\cal{C})_{/C}}(c, i(-))\colon\pst(\cal{A})_{/C} \to \cal{S}\] 
	is left Kan extended from its restriction to $\cal{A}_{/C} \subseteq \pst(\cal{A})_{/C}$, where $\pst(\cal{A})_{/C} := \pst(\cal{A}) \times_{\pst(\cal{C})} \pst(\cal{C})_{/C}$, and similarly for $\cal{A}_{/C}$.
\end{corollary}
\begin{proof}
Since $\pst(\cal{A})_{/C} \to \pst(\cal{A})$ is a right fibration and $\cal{A}_{/C} \to \cal{A}$ is its base change along $\cal{A} \to \pst(\cal{A})$ the comma $\infty$-categories of the inclusion $\cal{A}_{/C} \subseteq \pst(\cal{A})_{/C}$ are the same as those of $\cal{A} \subseteq \pst(\cal{A})$.
	By the pointwise formula for Kan extensions, we then need to show that for every map $A \to C$ with $A \in \pst(\cal{A})$, the natural map
	\[
		\underset{\cal{A}_{/A}}\colim\ \Map_{\cal{C}_{/C}}(c, i(-)) \to \Map_{\pst(\cal{C})_{/C}}(c,i(A))
	\]
	is an equivalence. 
	By universality of colimits in spaces, this is equivalent to the statement that for each $\phi\colon c \to i(A)$ in $\pst(\cal{C})_{/C}$ 
	the functor
    	\begin{align*}
		\cal{A}_{/A} \to \cal{S}, \quad (a \to A) \mapsto& \Map_{\pst(\cal{C})_{/C}}(c, i(a)) \times_{\Map_{\pst(\cal{C})_{/C}}(c, i(A))} \{\phi\} \\ 
		=& \Map_{\pst(\cal{C})}(c, i(a)) \times_{\Map_{\pst(\cal{C})}(c, i(A))} \{\phi\}
	\end{align*}
	has a contractible colimit. By~\cite[Corollary\ 3.3.4.6]{HTT}, it suffices to show that the unstraightening of this functor is weakly contractible. Indeed, this is exactly what is shown in Lemma~\ref{catoffact}.
\end{proof}

\begin{corollary}\label{cor:beck-chevalley}
Let $i\colon \cal{A} \hookrightarrow \cal{C}$ be a right-special inclusion and $\cal{E}$ some $\infty$-category which admits small colimits. 
Then the square
\[
\begin{tikzcd}[column sep = 50pt]
\cal{A} \ar[r,"i"]\ar[d,"j_{\cal{A}}"'] & \cal{C} \ar[d,"j_{\cal{C}}"] \\
\pst(\cal{A}) \ar[r,"\pst(i)"] & \pst(\cal{C}) 
\end{tikzcd}
\]
satisfies the following Beck-Chevalley properties:
\begin{enumerate}[leftmargin=*]
\item\label{item:functors}
For every functor $\vphi\colon \C \to \cal{E}$ the Beck-Chevalley transformation
\[ (j_{\cal{A}})_!i^*\vphi \Rightarrow \pst(i)^*(j_{\cal{C}})_!\vphi  \]
is an equivalence, where $(-)^*$ denotes restriction and $(-)_!$ left Kan extension.
\item\label{item:presheaves}
For every presheaf $\vphi\colon \pst(\cal{A})^{\op} \to \cal{E}$ the Beck-Chevalley transformation
\[ i_!j_{\cal{A}}^*\vphi \Rightarrow j_{\cal{C}}^*\pst(i)_!\vphi \]
is an equivalence, where by abuse of notation we denoted by $(-)^*$ and $(-)_!$ also restriction and left Kan extension of presheaves.
\end{enumerate}
\end{corollary}
\begin{proof}
For~\eqref{item:functors}, we claim that for every $A \in \pst(\cal{A})$ the induced functor on comma $\infty$-categories $\cal{A}_{/A} \to \C_{/A}$ (which appear in the pointwise formula for the relevant left Kan extensions) is cofinal, and for~\eqref{item:presheaves} we claim that for every $c \in \C$ the induced functor on comma categories $\cal{A}_{c/} \to \pst(\cal{A})_{c/}$ is coinitial; both statements are equivalent to the statement that for any map $c \to A$ in $\pst(\C)$ from an object of $\cal{C}$ to an object of $\pst(\cal{A})$ the $\infty$-category of factorizations $\cal{A}_{c//A}$ is weakly contractible, which we have established in Lemma~\ref{catoffact}.
\end{proof}

\begin{remark}\label{rem:statements-equiv}
In the proof of Corollary~\ref{cor:beck-chevalley} we gave an argument that directly implies both~\eqref{item:functors} and~\eqref{item:presheaves}. For the sake of completeness, let us point out that these statements are formally equivalent for any  square of $\infty$-categories; more precisely, the statement of~\eqref{item:functors} for a given $\cal{E}$ is equivalent to the statement of~\eqref{item:presheaves} for $\cal{E}^{\op}$. Indeed, starting from a commutative square with vertical restriction functors and horizontal left Kan extension functors we can pass to right adjoints everywhere to obtain a commutative square with vertical right Kan extension functors and horizontal restriction functors, and we can view right Kan extension of $\cal{E}$-valued functors as a left Kan extension of $\cal{E}^{\op}$-valued presheaves.
\end{remark}

\begin{remark}\label{rem:beck-chevalley}
In Corollary~\ref{cor:beck-chevalley}, passing to opposites turns right-special inclusions into left-special ones and replaces the prestable envelope $\pst(-) = \st(-)_{[0,\infty]}$ with the non-positive part $\st(-)_{[-\infty,0]}$ of the stable envelop. We hence conclude that if $i\colon \cal{A} \to \cal{C}$ is a left-special inclusion then the square
\[
\begin{tikzcd}[column sep = 50pt]
\cal{A} \ar[r,"i"]\ar[d,"j_{\cal{A}}"'] & \cal{C} \ar[d,"j_{\cal{C}}"] \\
\st(\cal{A})_{[-\infty,0]} \ar[r,"\st(i)_{[-\infty,0]}"] & \st(\cal{C})_{[-\infty,0]} 
\end{tikzcd}
\]
satisfies the Beck-Chevalley properties of Corollary~\ref{cor:beck-chevalley}, only that~\eqref{item:functors} holds for pre-sheaves and~\eqref{item:presheaves} for functors.
\end{remark}

The following is pre-stable version of Keller's theorem, which then immediately implies Keller theorem:
\begin{corollary}\label{cor:prestable-keller}
Let $i\colon \cal{A} \hookrightarrow \cal{C}$ be a right-special inclusion. Then the induced functor $\pst(\cal{A}) \to \pst(\cal{C})$ is fully faithful.
\end{corollary}
\begin{proof}
Applying Corollary~\ref{fact} with $c = i(a)$ for some $a \in \cal{A}$ one concludes that the induced map
\[ \Map_{\pst(\cal{A})}(a, A) \to \Map_{\pst(\cal{C})}(i(a),i(A)) \]
is an equivalence for every $A \in \pst(\cal{A})$. Since every object in $\pst(\cal{A})$ is generated under finite colimits by objects in $\cal{A}$ it follows that the map
\[ \Map_{\pst(\cal{A})}(B, A) \to \Map_{\pst(\cal{C})}(i(B),i(A)) \]
is an equivalence for every $B \in \pst(\cal{A})$.
\end{proof}

\begin{proof}[Proof of Theorem~\ref{thm:fully-faithful}]
For right-special inclusions apply Corollary~\ref{cor:prestable-keller} and use the fact that fully faithful functors induces fully faithful functors on Spanier-Whitehead stabilizations. For left-special inclusions pass to opposites.
\end{proof}

Having established that $\st(\cal{A}) \to \st(\cal{C})$ is fully faithful for left- or right-special inclusions, we now address the question of closure under retracts and behaviour with respect to the associated heart structures.

\begin{corollary}\label{retract}
	Let $\cal{A} \hookrightarrow \cal{C}$ be a retract-closed full inclusion of weakly idempotent complete exact categories which is either left or right-special. 
	Then the induced full inclusion $\st(\cal{A}) \hookrightarrow \st(\cal{C})$ is closed under retracts.
\end{corollary}
\begin{proof}
	Consider some retract diagram $X \to A \to X$ in $\st(\cal{C})$ with $A \in \st(\cal{A})$. By shifting, we may suppose $A \in \st(\cal{A})_{[0, n]}$ for some $n \ge 0$. We prove that $X \in \st(\cal{A})_{[0,n]}$ as well by induction on $n$. For $n=0$, this follows from the fact that $\cal{A}$ is closed under retracts in $\cal{C}$, and $\cal{C}$ under retracts in $\st(\cal{C})$ by Corollary~\ref{cor:retract-closure}.
For $n>0$, write $Y$ for the complementary summand to $X$ in $A$, which exists by weak idempotent completeness of $\st(\cal{C})$, and consider a heart decomposition
	\[
		a \xto{h} A \to B,
	\]
	for $a \in \cal{A}$ and $B \in \st(\cal{A})_{[1,n]}$. The map $h$ factors as $a \to a \oplus a \to X \oplus Y = A$, where the first map is the diagonal and the second the direct sum of the composites $f\colon a \to X$ and $g\colon a \to Y$ of $h$ with the canonical projections. The associated exact sequence of cofibres is then
	\[
		a \to B \to \cofib(f) \oplus \cofib(g),
	\]
	so that $\cofib(f)\oplus\cofib(g)$ is an extension of $B$ and $\Sigma a$, and hence lies in $\st(\cal{A})_{[1,n]}$. By the induction hypothesis $\cofib(f) \in \st(\cal{A})_{[1,n]}$. The exact sequence $a \to X \to \cofib(f)$ then show that $X$ belongs to $\st(\cal{A})_{[0, n]}$.
\end{proof}

We note the following useful feature of left- or right-special inclusions $\cal{A} \subseteq \cal{C}$ which are additionally closed under retracts. Call an exact functor of stable $\infty$-categories with heart structure $f\colon \cal{D} \to \cal{E}$ left (resp.\ right) \emph{heart conservative} if for $x \in \cal{D}$ with $f(x) \in \cal{E}_{\le 0}$ (resp.\ $f(x) \in \cal{E}_{\ge 0}$), we have $x \in\cal{D}_{\le 0}$ (resp.\ $x \in\cal{D}_{\ge 0}$). The functor $f$ is heart conservative if it is both left and right heart conservative.

\begin{corollary}\label{conservativity}
	Let $\cal{A} \hookrightarrow \cal{C}$ is a right-special retract-closed inclusion of exact categories. Then the induced fully faithful functor $\st(\cal{A}) \hookrightarrow \st(\cal{C})$ is left heart-conservative. In particular, if $\cal{C}$ is weakly idempotent complete and $\cal{A} \hookrightarrow \cal{C}$ is a Keller inclusion then $\st(\cal{A}) \cap \cal{C} = \cal{A}$.
\end{corollary}
\begin{proof}
Suppose given some $X \in \st(\cal{A}) \cap \st(\cal{C})_{\le 0}$. Take a heart decomposition 
\[ X' \to X \to X''\] 
in $\st(\cal{A})$ with $X' \in \st(\cal{A})_{\le -1}$ and $X'' \in \st(\cal{A})_{\ge 0}$. Then the image of $X''$ in $\st(\cal{C})$ lies in $\st(\cal{A})_{\ge 0} \cap \st(\cal{C})_{\le 0} = \cal{C} \cap \st(\cal{A})_{\ge 0}$. In particular, the identity $\id_{X''}$ is a map from an object of $\cal{C}$ to an object of $\st(\cal{A})_{\ge 0}$, and so factors through an object of $\cal{A}$ by Lemma~\ref{catoffact}. But $\cal{A}$ is closed under retracts in $\cal{C}$, so $X'' \in \cal{A}$, and $X$ is an extension of objects in $\st(\cal{A})_{\le 0}$ and itself in $\st(\cal{A})_{\le 0}$.
\end{proof}

\subsection{Calculus of fractions}\label{sec:calculus}

Throughout this section let us fix an extension-closed full inclusion $\cal{A} \to \cal{C}$. We denote by $\cal{C}^{\cal{A}}_{\inc}$ and $\cal{C}^{\cal{A}}_{\pr}$ the wide subcategories of $\cal{C}$ whose morphisms are the inclusions with cofibre in $\cal{A}$ and projections with fibre in $\cal{A}$, respectively (where we note that the extension-closed condition assures that these classes of maps are closed under composition). We then consider the associated localisations (ignoring exact structures)
\[ \cal{B}^{l} := \cal{C}[(\cal{C}^{\cal{A}}_{\inc})^{-1}] \quad\text{and}\quad \cal{B}^{r} := \cal{C}[(\cal{C}^{\cal{A}}_{\pr})^{-1}] \]
and write $p^l\colon \cal{C} \to \cal{B}^{l}$ and $p^r\colon \cal{C} \to \cal{B}^r$ for the associated localisation functors. In this section we develop some basic properties of these localisations related to their associated calculus of fractions.

\begin{lemma}\label{lem:calculus}
The localisation $p^l\colon \cal{C} \to \cal{B}^l$ admits a left calculus of fractions, that is, for $x,y \in \cal{C}$ the map
\[ \colim_{[y \mono y'] \in y\rtimes \cal{A}}\Map_{\cal{C}}(x,y') \to \Map_{\cal{B}^l}(p(x),p(y)) \]
is an equivalence, where $y\rtimes \cal{A} \subseteq \cal{C}_{y/}$ is the full subcategory spanned by the inclusions $y \mono y'$ with cofibre in $\cal{A}$. 
In particular, if $f\colon x \to y$ is a map in $\cal{B}^l$ then for every $\tilde x\in \cal{C}$ lifting $x$ we can lift $f$ to a map $\tilde f\colon \tilde x \to \tilde y$ with domain $\tilde x$. Dually, the localisation $p^r\colon \cal{C} \to \cal{B}^r$ admits a right calculus of fractions and every map in $\cal{B}^r$ admits a lift in $\cal{C}$ given any prescribed lift of its codomain.
\end{lemma}
\begin{proof}
Apply \cite[Theorem 7.2.16 and 7.2.8]{Cis19} (see also~\cite[Corollary 3.14]{nuiten16}) using that the collection of inclusions with cofibre in $\cal{A}$ is closed under composition (since $\cal{A}$ is closed under extensions) and cobase change.
\end{proof}

\begin{lemma}\label{lem:lift-inclusions}
Suppose that the localisation $p^l\colon \cal{C} \to \cal{B}^l$ sends cocartesian squares in $\cal{C}$ with parallel maps inclusions to cocartesian squares in $\cal{B}^l$.
If $f\colon x \to y$ in $\cal{B}^l$ is the image of some inclusion in $\cal{C}$ then for every $\tilde x\in \cal{C}$ lifting $x$ we can lift $f$ to an inclusion $\tilde f\colon \tilde x \mono \tilde y$ with domain $\tilde x$. 

Dually, if $p^r\colon \cal{C} \to \cal{B}^r$ sends cartesian squares in $\cal{C}$ with parallel maps projections to cartesian squares in $\cal{B}^r$ then any map in $\cal{B}^{r}$ which is the image of some projection in $\cal{C}$ can be lifted to a projection in $\cal{C}$ with any prescribed lift of its codomain.
\end{lemma}
\begin{proof}
Suppose given a map $f\colon x \to y$ in $\cal{B}^l$ which is the image of an inclusion $f'\colon x'\mono y'$ in $\cal{C}$ and let $\tilde x \in \cal{C}$ be a lift of $x$. 
Then the equivalence $p(\tilde x) \simeq p(x')$ can be represented by a zig zag $x' \rightarrow w \hookleftarrow \tilde x$ with the second arrow an inclusion with cofibre in $\cal{A}$ and the first arrow mapping to an equivalence in $\cal{B}^l$. We may then consider the cobase change 
\[
\begin{tikzcd}
& x' \ar[r,"f'",hookrightarrow ]\ar[d] & y'\ar[d] \\
\tilde x \ar[r,hookrightarrow] & w \ar[r, "f''",hookrightarrow] & z 
\end{tikzcd}
\]
of $f'$. The square on the right is then also cocartesian in $\cal{B}^l$ by assumption, and so its right vertical arrow is also an equivalence in $\cal{B}^l$. The composite inclusion $\tilde x \mono w \mono z$ then provides the desired lift of $f$.
\end{proof}

\begin{proposition}\label{prop:preserves-cocartesian}
If $\cal{A} \to \cal{C}$ is left-special then $p^l\colon \cal{C} \to \cal{B}^l$ sends cocartesian square in $\cal{C}$ with parallel legs inclusions to cocartesian squares in $\cal{B}^l$. Dually, if $\cal{A} \to \cal{C}$ is right-special then $p^r\colon \cal{C} \to \cal{B}^r$ sends cartesian square in $\cal{C}$ with parallel legs projections to cartesian squares in $\cal{B}^r$. In particular, in these cases the lifting property of Lemma~\ref{lem:lift-inclusions} holds.
\end{proposition}
\begin{proof}
We prove the right-special case, the left-special case can be deduced by passing to opposite exact $\infty$-categories.
We view the pre-stable envelope $\pst(\cal{C})$ as an exact $\infty$-category in which all maps are inclusions. Given $y \in \cal{C}$, we claim that for every $x \in \cal{C}$, the functor
\[ y\rtimes \pst(\cal{A}) \to \Sps \quad\quad [y \mono y'] \mapsto \Map_{\pst(\cal{C})}(x,y') \]
is left Kan extended from its restriction to $y \rtimes \cal{A}$. Indeed, viewing $\Sigma y$ as an object of $\st(\cal{C})$, note that the functor $[y \mono y'] \mapsto \cofib[y \mono y'] \to \Sigma y$ induces an equivalence 
\[ y\rtimes \cal{A} \simeq \cal{A}_{/\Sigma y} = \cal{A} \times_{\st(\cal{C})} \st(\cal{C})_{/\Sigma y},\] 
and similarly for $y\rtimes \pst(\cal{A})$. In addition, under these equivalences the functor $[y \mono y'] \mapsto \Map_{\pst(\cal{C})}(x,y')$ corresponds to the functor $[z \to \Sigma y] \mapsto \Map_{\cal{C}_{/\Sigma y}}(x,z)$, where $x$ is viewed as the object of $\cal{C}_{/\Sigma y}$ corresponds to the zero map $x \xrightarrow{0} \Sigma y$. The left Kan extension claim is then an instance of Corollary~\ref{fact}. Now since $\pst(\cal{C})$ admits finite colimits and the inclusion $\pst(\cal{A}) \to \pst(\cal{C})$ preserves finite colimits it follows that $y\rtimes \pst(\cal{A})$ admits finite colimits, and is hence filtered. We conclude that the presheaf
\[ \Map_{\cal{C}'}(-,p(y)) \simeq \colim_{[y \to y'] \in y\rtimes \cal{A}} \Map_{\cal{C}}(-,y') \simeq \colim_{[y \to y'] \in y\rtimes \pst(\cal{A})} \Map_{\pst(\cal{C})}(-,y') \]
is a filtered colimit of presheaves of the form $\Map_{\pst(\cal{C})}(-,y')$. Now since the inclusion $\cal{C} \to \pst(\cal{C})$ is an exact functor every cocartesian square in $\cal{C}$ with parallel legs inclusions is also cocartesian in $\pst(\cal{C})$, and hence each of the presheaves $\Map_{\pst(\cal{C})}(-,y')$ send such a square to a pullback square. We hence conclude that $\Map_{\cal{C}'}(-,p(y))$ sends such squares to pullback squares. It follows that $p^l$ sends cocartesian square in $\cal{C}$ with parallel legs inclusions to cocartesian squares in $\cal{B}^{l}$, as desired.
\end{proof}

\subsection{Keller sequences}\label{subsec:keller-sequences}

In this section we show that the class of Keller inclusions yields a well-behaved class of fibre-cofibre sequences 
\[ \cal{A} \to \cal{C} \to \cal{B} \]
in $\exact$, which we call \emph{Keller sequences}. The arrow $\cal{C} \to \cal{B}$ is then called a Keller projection, and admit explicit characterization. In particular, we show that every Keller inclusion admits a cofibre in $\exact$ which is a Keller projection, and is furthermore the fibre of its cofibre. Similarly, the fibre of every Keller projection is a Keller inclusion and every Keller projection is the cofibre of its fibre. This is analogous to the behaviour of Verdier sequences of stable $\infty$-categories. 

In what follows, by an \emph{exact localisation functor} $p\colon \cal{C} \to \cal{B}$ we will mean an exact functor whose underlying functor is a Dwyer-Kan localisation and such that a map in $\cal{B}$ is an inclusion resp.\ projection if and only if it is the essential image of an inclusion resp.\ projection under $p$.
\begin{definition}\label{specproj}
	A map $p:\cal{C} \to \cal{B}$ in $\exact$ is a \emph{left-special projection} if the following hold:
	\begin{enumerate}[label=(\roman*)]
		\item\label{iLR} $p$ is an exact localisation functor.
	\end{enumerate}
	\begin{enumerate}[label=(\roman*.L)]
		\setcounter{enumi}{1}
		\item\label{iiL} for every map $f\colon x \to y$ in $\cal{C}$ such that $p(f)$ is an equivalence there exists an $a \in \ker(p)$	 and a factorisation
		\[\begin{tikzcd}
		& x \oplus a \ar[rd, "q", twoheadrightarrow] \\
		x \ar[ru, rightarrowtail] \ar[rr, "f"] && y,
		\end{tikzcd}\]
		such that $q$ is a projection and the left diagonal map is the corresponding summand inclusion. 
	\end{enumerate}
	Dually, $p\colon\cal{C} \to \cal{B}$ is a \emph{right-special projection} if \ref{iLR} holds, and:
	\begin{enumerate}[label=(\roman*.R)]
		\setcounter{enumi}{1}
		\item\label{iiR} for each map $f\colon x \to y$ in $\cal{C}$ such that $p(f)$ is an equivalence there exists an $b \in \ker(p)$ and a factorisation
		\[\begin{tikzcd}
			& y \oplus b \ar[rd, twoheadrightarrow] \\
			x \ar[rr, "f"] \ar[ru, "j", rightarrowtail] && y,
		\end{tikzcd}\]
		such that $j$ is an inclusion and right diagonal map is the corresponding summand projection.  	\end{enumerate}
	Finally, $p\colon \cal{C} \to \cal{B}$ is a \emph{Keller projection} if it is both left and right-special.
\end{definition}

\begin{remark}\label{rem:calculus}
Let $p\colon \cal{C} \to \cal{B}$ be an exact localisation functor and set $\cal{A} = \ker(p) \subseteq \cal{C}$, viewed as an extension-closed exact subcategory of $\cal{C}$ with the induced exact structure. 
If $p$ is a left-special projection then $p$ also exhibits $\cal{B}$ as the localisation of $\cal{C}$ by $\cal{C}^{\cal{A}}_{\pr}$, and hence $p$ enjoys a right calculus of fraction as in Lemma~\ref{lem:calculus}. In addition, projections in $\cal{B}$ lift to projections in $\cal{C}$ with prescribed lifts of their codomains by Lemma~\ref{lem:lift-inclusions}. Dually, 
If $p$ is a right-special projection then $p$ also exhibits $\cal{B}$ as the localisation of $\cal{C}$ by $\cal{C}^{\cal{A}}_{\inc}$, enjoys the associated left calculus of fraction, and inclusions in $\cal{B}$ lift to inclusions in $\cal{C}$ with prescribed lift of their domain.
In particular, when $p$ is a Keller projection it enjoys simultaneously all the above properties.
\end{remark}

\begin{example}\label{ex:verdier-projection}
An exact functor $\cal{C} \to \cal{D}$ among stable $\infty$-categories is a Keller projection with respect to the respective maximal exact structures if and only if it is a Verdier projection. 
\end{example}

\begin{proposition}\label{prop:verdier}
	Given a nullcomposite sequence
	\[\begin{tikzcd}
		\cal{A} \ar[r] \ar[d] & \cal{C} \ar[d] \\
		0 \ar[r] & \cal{B}
	\end{tikzcd}\]
	of exact categories, the following are equivalent:
	\begin{enumerate}[label=(\roman*)]
		\item\label{cof} $i$ is a Keller inclusion and $p\colon \cal{C} \to \cal{B}$ is the cofibre of $i$ in $\exact$;
		\item\label{fib} $p$ is a Keller projection and $i\colon\cal{A} \to \cal{C}$ is the kernel of $p$ in $\exact$.
	\end{enumerate}
\end{proposition}

We will refer to a sequence satisfying one of the equivalent conditions of Proposition~\ref{prop:verdier} as a \emph{Keller sequence}.

\begin{remark}
A Keller sequence among stable $\infty$-categories is just a Verdier sequence.
\end{remark}

The following proposition establishes half of Proposition~\ref{prop:verdier}.

\begin{proposition}\label{prop:keller-cof}
Let $i\colon \cal{A} \hookrightarrow \cal{C}$ be a Keller inclusion among weakly idempotent complete exact $\infty$-categories. Let $W$ denote the collection of maps in $\cal{C}$ whose cofibre in $\st(\cal{C})$ lies in $\st(\cal{A})$. Then the localisation $\cal{B} := \cal{C}[W^{-1}]$ inherits an exact structure in which a map is an inclusion (resp. projection) if and only if it is the image of an inclusion (resp. projection) in $\cal{C}$, and the resulting projection $\cal{C} \to \cal{B}$ is exact and a Keller projection. Finally, the resulting null-composite sequence
\[ \cal{A} \xrightarrow{i} \cal{C} \xrightarrow{p} \cal{B} \]
is both cartesian and cocartesian. In particular, Keller inclusions admit cofibres which are Keller projections, and are furthermore the fibres of said projections. 
\end{proposition}
\begin{proof}
Note first that since the subcategory $W$ is closed under direct sum the localisation $\cal{B}$ is additive and the map $\cal{C} \to \cal{B}$ is direct sum preserving.
Let now $f\colon x \to y$ be a map in $W$, so that the cofibre $A = \cofib[f] \in \st(\cal{C})$ lies in $\st(\cal{A})$, and hence in $\st(\cal{A})_{[0,1]}$ by Corollary~\ref{conservativity}. We claim that there exists a $b \in \cal{A}$ such that $f$ factors as a composite $x \mono b \oplus y \epi y$ with the second map being the projection to $y$. To see this, note that by Lemma~\ref{lem:factor}
the map $y \to A$ factors as $y \to a \to A$ with $a \in \cal{A}$ and $\cofib[a \to A] = \Sigma b$ for some $b \in \cal{A}$. 
Let $z$ be the fibre product in the right square in
\[
\begin{tikzcd}
x\ar[r]\ar[d,equal] & z \ar[r]\ar[d] & a \ar[d] \\
x\ar[r] & y \ar[r] & A
\end{tikzcd}
\]
where the map $x \to z$ is induced by the universal property pullback from the map $f\colon x \to y$ and the zero map $0\colon x \to a$. Then $\fib[z \to y] \simeq b \in \cal{A}$ and so $z$ belongs to $\cal{C}$ and the map $z \to y$ is a projection in $\cal{C}$. At the same time, the cofibre of $x \to z$ in $\st(\cal{C})$ is $a$, and so $x \to z$ is an inclusion in $\cal{C}$.
Now since the map $y \to A$ factors through $a$ by construction the map $z \to y$ admits a section and hence splits as the summand projection $y \oplus b \to y$. Applying this argument to the opposite exact $\infty$-categories we similarly get that any $f\colon x \to y$ in $W$ factors as a composite $x \mono x \oplus a \epi y$, where the first map is the associated summand inclusion.

We conclude that the projection $\cal{C} \to \cal{B}$ also exhibits $\cal{B}$ as the localisation of $\cal{C}$ by $\cal{C}^{\cal{A}}_{\inc}$, and also as the localisation of $\cal{C}$ by $\cal{C}^{\cal{A}}_{\pr}$. In particular, by Lemma~\ref{lem:calculus} this localisation has both a left calculus of fractions with respect to $\cal{C}^{\cal{A}}_{\inc}$ and a right calculus of fractions with respect to $\cal{C}^{\cal{A}}_{\pr}$. Furthermore, by Lemma~\ref{prop:preserves-cocartesian}
the map $p\colon \cal{C} \to \cal{B}$ preserves cocartesian squares with parallel legs inclusions and cartesian squares with parallel legs projections. Using that the axioms of exact structures are satisfied in $\cal{C}$, to show that under the proposed collection of inclusions and projections the $\infty$-category $\cal{B}$ and the functor $p$ are exact it will hence suffice to show that any of the following types of diagrams
\[
\begin{tikzcd}
x \ar[r,hookrightarrow]\ar[d] & y \\
z &
\end{tikzcd}
\quad\quad
\begin{tikzcd}
x \ar[r,hookrightarrow]\ar[d,->>] & y \\
z &
\end{tikzcd}
\quad\quad
\begin{tikzcd}
& y\ar[d,->>] \\
z\ar[r] & w
\end{tikzcd}
\quad\quad
\begin{tikzcd}
& y\ar[d,->>] \\
z\ar[r,hookrightarrow] & w
\end{tikzcd}
\]
in $\cal{B}$ lifts to a diagram of the same type in $\cal{C}$. All of these follow from Lemmas \ref{lem:calculus} and \ref{lem:lift-inclusions}. Indeed, since we assume that $\cal{A} \hookrightarrow \cal{C}$ is both left- and right-special these lemmas together tell us that we can lift any map in $\cal{B}$ to a map in $\cal{C}$ given a prescribed lift of either its domain or its target, that we can lift any inclusion in $\cal{B}$ to an inclusion in $\cal{C}$ given a prescribed lift of its domain, and that we can lift any projection in $\cal{B}$ to a projection in $\cal{C}$ given a prescribed lift of its target. To obtain the first two diagram we first lift $z$ to $\tilde z$, then lift the map $x \to z$ to map $\tilde x \to \tilde z$ using the right calculus of fractions (which, in the second case we can choose to be a projection), and then lift the inclusion $x \mono y$ to an inclusion $\tilde x \to \tilde y$. For the latter two diagrams, we first lift $z$ to $\tilde z$, then lift the map $z \to w$ to a map $\tilde z \to \tilde w$ (which in the second case we can choose to be an inclusion), and then lift $z \epi w$ to a projection $\tilde z \epi \tilde w$. We have thus established that $\cal{B}$ is an exact $\infty$-category with respect to the proposed structure and that $p\colon \cal{C} \to \cal{B}$ is exact. 

We now claim that the resulting null-composite sequence
\begin{equation}\label{eq:keller-sequence} 
\cal{A} \xrightarrow{i} \cal{C} \xrightarrow{p} \cal{B} 
\end{equation}
is both cartesian and cocartesian. For this, note first that by the argument above, the cocartesian squares with parallel legs inclusions in $\cal{B}$ are the exactly the images of the same types of squares in $\cal{C}$, and similarly for cartesian squares with parallel legs projections. The universal property of localisations then implies that for any exact $\infty$-category $\cal{E}$, the induced functor
\[ \funex(\cal{B},\cal{E}) \to \funex(\cal{C},\cal{E}) \]
is fully faithful with essential image those exact functors $\cal{C} \to \cal{E}$ which send the arrows in $W$ to equivalences. Clearly any such exact functor vanishes on $\cal{A}$ (since the maps from $0$ to/from any object in $\cal{A}$ is in $W$), but on the other hand since we have shown in the beginning of the proof that any map in $W$ is a composite of an inclusion with cofibre in $\cal{A}$ and a projection with fibre in $\cal{A}$ we see that any exact functor that vanishes on $\cal{A}$ also inverts all arrows in $W$. Comparing universal properties we hence conclude that the sequence~\eqref{eq:keller-sequence} exhibits $\cal{B}$ as the cofibre of $\cal{A} \hookrightarrow \cal{C}$.

Now since forming stable envelops is a colimit preserving operation we have $\st(\cal{B}) = \st(\cal{C})/\st(\cal{A})$. At the same time, by Corollary~\ref{retract} the induced functor $\st(\cal{A}) \to \st(\cal{C})$ is fully faithful and closed under retracts, and so the map $\st(\cal{A}) \to \fib[\st(\cal{C}) \to \st(\cal{B})]$ is an equivalence by \cite[Corollary A.1.10]{9-authors-II}. Since $\cal{C}$ embeds in $\st(\cal{C})$ and $\cal{A} = \cal{C} \cap \st(\cal{A})$ by Corollary~\ref{conservativity} we now conclude that the sequence~\eqref{eq:keller-sequence} is a fibre sequence in $\cat$, and eventually in $\exact$ by Corollary~\ref{cor:detects}.

It is left to show that $p$ is a Keller projection. Note that since the fibre of $p$ was shown to coincide with $\cal{A}$, the factorization properties we proved for $W$ in the beginning of the proof imply that $p$ satisfies~\ref{iiR} and~\ref{iiL} as soon as we show that $W$ is saturated, that is, that every map $f\colon x \to y$ such that $p(f)$ is an equivalence is in $W$. For this, note that if $p(f)$ is an equivalence then $\cofib[f] \in \st(\cal{C})$ lies in the fibre of $\st(\cal{C}) \to \st(\cal{B})$. But we have already saw above that this fibre is $\st(\cal{A})$, and so the proof is complete.
\end{proof}

\begin{proof}[Proof of Proposition \ref{prop:verdier}]
The implication \ref{cof} $\Rightarrow$ \ref{fib} follows from Proposition~\ref{prop:keller-cof}.
We now show that \ref{fib} $\Rightarrow$ \ref{cof}. 
For this we first show that if $p$ is a left-special projection with kernel $i$ then $i$ is a left-special inclusion (and hence, passing to opposite exact $\infty$-categories, that if $p$ is a right-special projection with kernel $i$ then $i$ is a right-special inclusion). 
	Given a projection $x \epi a$ in $\cal{C}$ with $a \in \cal{A}$, write $f:y \mono x$ for its fibre. Since $p$ is a left-special projection, 	there exists a $b \in \cal{A}$ and factorization 	\[\begin{tikzcd}
		& y \oplus b \ar[rd, "q", twoheadrightarrow] \\
		y \ar[rr, rightarrowtail, "f"] \ar[ru, rightarrowtail] && x
	\end{tikzcd}\]
	where $q$ is a projection with fibre in $\cal{A}$ and the left diagonal map is the corresponding summand inclusion.  
	We claim that the composite $b \to y\oplus b \epi x \epi a$ is a projection in $\cal{A}$. Indeed, there is a map of exact sequences
	\[\begin{tikzcd}
		y \ar[r, rightarrowtail] \ar[d, equal] & y \oplus b \ar[r, twoheadrightarrow] \ar[d, twoheadrightarrow] \ar[r] & b \ar[d] \\
		y \ar[r, rightarrowtail, "f"] & x \ar[r, twoheadrightarrow] & a 
	\end{tikzcd}\]
	where the upper row is split. Since the right square induces an equivalence on horizontal fibres it is cartesian in $\st(\cal{C})$, and so also induces an equivalence on vertical fibres in $\st(\cal{C})$. Since $q$ is a projection these vertical fibres are in $\cal{C}$ and $b \to a$ is a projection in $\cal{C}$. Furthermore, since the fibre of $q$ is assumed in $\cal{A}$ we have that the fibre of $b \to a$ is in $\cal{A}$, so that $b \to a$ is a also a projection in $\cal{A}$. 
 
	Now since we assume that $\cal{A} \to \cal{C}$ is the inclusion of the kernel of $p$ it is closed under retracts, and hence a Keller inclusion.  
It is left to show that the sequence exhibits $\cal{B}$ as the cofibre of $\cal{A} \to \cal{C}$. Let $W$ denote the collection of maps sent to equivalences by $p$. Since the underlying functor of $p$ is a localisation by $W$ and the inclusions/projections in $\cal{B}$ are exactly the images of the inclusions/projections in $\cal{C}$ we have that for any exact $\infty$-category $\cal{E}$ the induced functor
\[ \funex(\cal{B},\cal{E}) \to \funex(\cal{C},\cal{E}) \]
is fully faithful with essential image those exact functors $\cal{C} \to \cal{E}$ which send the arrows in $W$ to equivalences. Now since the arrow from $0$ to (or from) any object in $\cal{A}$ is in $W$ we have that such exact functors vanish on $\cal{A}$. On the other hand, either of the axioms~\ref{iiL}, \ref{iiR} implies that any exact functor $\cal{C} \to \cal{E}$ which vanishes on $\cal{A}$ inverts all the arrows in $W$. Comparing universal properties we conclude that $p$ exhibits $\cal{B}$ as the cofibre of $i\colon \cal{A} \to \cal{C}$, as desired.
\end{proof}

\begin{corollary}\label{st}
	$\st(-)$ sends Keller sequences to Verdier sequences in $\catex$.
\end{corollary}
\begin{proof}
	Given a Keller sequence $(i,p)\colon\cal{A} \to \cal{C} \to \cal{B}$, the induced map $\st{p}\colon\st(\cal{C}) \to \st(\cal{B})$ is a Verdier projection by Proposition \ref{prop:verdier} since $\st(-)$ is left adjoint, with kernel the retract-closure of $\st(\cal{A})$ in $\st(\cal{C})$, which by Corollary \ref{retract} coincides with $\st(\cal{A})$.
\end{proof}

\begin{corollary}\label{cor:heartwise-keller}
Let $\cal{A} \to \cal{C} \to \cal{B}$ be a Keller sequence. Then for every $-\infty \leq m \leq n \leq \infty$ the sequence
\[ \st(\cal{A})_{[m,n]} \to \st(\cal{C})_{[m,n]} \to \st(\cal{B})_{[m,n]} \]
is a Keller sequence.
\end{corollary}
\begin{proof}
The map on the left is fully faithful by Theorem~\ref{thm:fully-faithful} and closed under retracts 
by Corollaries~\ref{conservativity} and~\ref{retract}.
Now if $m=-\infty$ and $n=\infty$ then this is Corollary~\ref{st} (see Examples \ref{ex:verdier-inclusion} and \ref{ex:verdier-projection}). If $n=\infty$ and $m \neq \infty$ then we may as well assume that $m=0$. Then $\st(\cal{A})_{[0,\infty]} = \pst(\cal{A}) \to \pst(\cal{C}) = \st(\cal{C})_{[0,\infty]}$ is left-special by Lemma~\ref{lem:square-left-special} and is right-special since all maps are inclusions on both sides. By Proposition~\ref{prop:verdier} we then have that the map $\pst(\cal{C}) \to \pst(\cal{C})/\pst(\cal{A})$ to the cofibre in $\exact$ is a Keller projection. Since all maps in $\pst(\cal{C})$ are inclusions it then follows that all the maps in $\pst(\cal{C})/\pst(\cal{A})$ are inclusions, so that this cofibre is also the cofibre in the full subcategory $\exact^{\pst} \subseteq \exact$ spanned by the exact $\infty$-categories in which all maps are inclusions, which is just the $\infty$-category of pre-stable $\infty$-categories and finite colimit preserving functors. Since $\pst(-)$ constitutes a left adjoint to the inclusion $\exact^{\pst} \subseteq \exact$ we conclude that the functor $\pst(\cal{C})/\pst(\cal{A}) \to \pst(\cal{C}/\cal{A}) = \pst(\cal{B})$ is an equivalence, and so
\[ \pst(\cal{A}) \to \pst(\cal{B}) \to \pst(\cal{C}) \]
is a Keller sequence. The case where $m=-\infty$ and $n\neq \infty$ follows in a dual manner.

Finally, consider the case where $m \neq -\infty$ and $n \neq \infty$. Without loss of generality we may assume that $m=0$. Applying both Lemma~\ref{lem:square-left-special} and its dual we see that $\st(\cal{A})_{[0,n]} \to \st(\cal{C})_{[0,n]}$ is a Keller inclusion. To finish the proof we need to show that the induced exact functor
\begin{equation}\label{eq:induced-cofibre}
\st(\cal{C})_{[0,n]}/\st(\cal{A})_{[0,n]} \to \st(\cal{B})_{[0,n]} 
\end{equation}
is an equivalence. Now since the inclusion $\cal{C} \to \st(\cal{C})_{[0,n]}$ induces an equivalence on stable envelopes by~\cite[Theorem 2.9]{Sau23} (and its proof), the inclusion $\st(\cal{C})_{[0,n]} \subseteq \st(\cal{C})$ exhibits its target as the stable envelope of its source, and since taking stable envelopes preserves cofibres we conclude that~\eqref{eq:induced-cofibre} induces an equivalence on stable envelopes. The exact functor~\eqref{eq:induced-cofibre} is hence an extension-closed inclusion by~\cite[Proposition 4.25]{Kle22}. It will hence suffice to show that $\st(\cal{C})_{[0,n]} \to \st(\cal{B})_{[0,n]}$ is essentially surjective. 
We now argue by induction on $n$. For $n=0$ this is clear. Now suppose that $n$ is such that the claim is true for $n-1$. Then $\st(\cal{C})_{[0,n-1]} \to \st(\cal{B})_{[0,n-1]}$ is a Keller projection and in particular every map $f\colon x \to y$ in $\st(\cal{B})_{[0,n-1]}$ is the image of a map in $\st(\cal{C})_{[0,n-1]}$ in the sense that it lifts to a zig-zag $\tilde{x} \to w \leftarrowtail \tilde{y}$ with right-hand map an inclusion with cofibre in $\st(\cal{A})_{[0,n-1]}$ (so the image of the left-hand map in $\st(\cal{B})_{[0,n-1]}$ is equivalent to $f$). Since any object in $\st(\cal{B})_{[0,n]}$ is a cofibre of a map in $\st(\cal{B})_{[0,n-1]}$ it now follows that $\st(\cal{C})_{[0,n]} \to \st(\cal{B})_{[0,n]}$ is essentially surjective, as desired.
\end{proof}

The following corollary shows that in the definition of Keller projections, Properties \ref{iiL} and \ref{iiR} could be strengthened so that the \ref{iiL} applies to all maps in $\cal{C}$ lying over projections in $\cal{B}$ and \ref{iiR} applies to all maps in $\cal{C}$ lying over inclusions in $\cal{B}$, giving an a priori stronger notion but a posteriori equivalent notion of Keller projections:

\begin{corollary}\label{cor:stronger-ii}
Let $p\colon \cal{C} \to \cal{B}$ be an exact localisation functor among weakly idempotent complete $\infty$-categories and set $\cal{A} =: \ker(p)$. If $p$ is a left-special projection then for every map $f\colon x \to y$ in $\cal{C}$ such that $p(f)$ is a projection there exists an $a \in \cal{A}$ and a factorisation
\[\begin{tikzcd}
& x \oplus a \ar[rd, "q", twoheadrightarrow] \\
x \ar[ru, rightarrowtail] \ar[rr, "f"] && y,
\end{tikzcd}\]
such that $q$ is a projection and the left diagonal map is the corresponding summand inclusion. Dually, if $p$ is a right-special projection then for every map $f\colon x \to y$ in $\cal{C}$ such that $p(f)$ is an inclusion there exists an $a \in \cal{A}$ and a factorisation
\[\begin{tikzcd}
& y \oplus b \ar[rd, twoheadrightarrow] \\
x \ar[ru, rightarrowtail,"j"] \ar[rr, "f"] && y,
\end{tikzcd}\]
such that $j$ is an inclusion and right diagonal map is the corresponding summand projection. 
\end{corollary}
\begin{proof}
We prove the claim for left-special projections. The case of right-special projections follows by taking opposite exact $\infty$-categories. Let $f\colon x \to y$ be a map such that $p(f)$ is a projection and let $C = \cofib[f] \in \st(\cal{C})$, so that $C$ lies in $\st(\cal{C})_{[0,1]}$. Since $p(f)$ is a projection we have $p(C) \in \Sigma\cal{B}$ belongs to $\cal{B}$. It follows that the map $0 \to p(C)$ is a projection in $\st(\cal{B})_{[0,1]}$. 
Since $\st(\cal{C})_{[0,1]} \to \st(\cal{B})_{[0,1]}$ is a Keller projection with kernel $\st(\cal{A})_{[0,1]}$ by Corollary~\ref{cor:heartwise-keller} it follows from Remark~\ref{rem:calculus} that there exists a projection $A \epi C$ in $\st(\cal{C})_{[0,1]}$ with $A \in \st(\cal{A})_{[0,1]}$. Choose a heart decomposition $b \to A \to \Sigma b'$ with $b,b' \in \cal{A}$. Then $b \to A$ is also a projection in $\st(\cal{A})_{[0,1]}$ and so the composite $b \epi A \epi C$ is a projection in $\st(\cal{C})_{[0,1]}$. Let $z$ be the fibre product in the right square in
\[
\begin{tikzcd}
x \ar[d,equal]\ar[r] & z \ar[r]\ar[d,->>] & b \ar[d,->>] \\
x\ar[r] & y \ar[r] & C
\end{tikzcd}
\]
Then the map $z \epi y$ is a projection in $\st(\cal{C})_{[0,1]}$ with fibre $b' \in \cal{A} \subseteq \cal{C}$ and hence $z$ belongs to $\cal{C}$ and $z \epi y$ is a projection in $\cal{C}$. The cofibre of $x \to z$ is $b$ by construction and so this map is in particular an inclusion in $\cal{C}$ which is sent to an equivalence in $\cal{B}$. By Property~\eqref{iiL} there exists an $a\in \cal{A}$ and a factorization of $x \to z$ as $x \mono x \oplus a \epi z$ with the first map being the summand inclusion. The triangle
\[\begin{tikzcd}
& x \oplus a \ar[rd, "q", twoheadrightarrow] \\
x \ar[ru, rightarrowtail] \ar[rr, "f"] && y,
\end{tikzcd}\]
hence gives the desired factorization of $f$.
\end{proof}

We end this section by recording the following version of Corollary~\ref{cor:beck-chevalley} for Keller projections:
\begin{corollary}\label{cor:beck-chevalley-2}
Let $p\colon \cal{C} \hookrightarrow \cal{B}$ be a Keller projection and $\cal{E}$ some $\infty$-category which admits small colimits. 
Then the square
\[
\begin{tikzcd}[column sep = 50pt]
\cal{C} \ar[r,"p"]\ar[d,"j_{\cal{C}}"'] & \cal{B} \ar[d,"j_{\cal{B}}"] \\
\pst(\cal{C}) \ar[r,"\pst(p)"] & \pst(\cal{B}) 
\end{tikzcd}
\]
satisfies the following Beck-Chevalley properties:
\begin{enumerate}[leftmargin=*]
\item\label{item:functors-2}
For every functor $\vphi\colon \cal{B} \to \cal{E}$ the Beck-Chevalley transformation
\[ (j_{\cal{C}})_!p^*\vphi \Rightarrow \pst(p)^*(j_{\cal{B}})_!\vphi  \]
is an equivalence. 
\item\label{item:presheaves-2}
For every presheaf $\vphi\colon \pst(\cal{C})^{\op} \to \cal{E}$ the Beck-Chevalley transformation
\[ p_!j_{\cal{C}}^*\vphi \Rightarrow j_{\cal{B}}^*\pst(p)_!\vphi \]
is an equivalence. 
\end{enumerate}
\end{corollary}
\begin{proof}
The statements~\eqref{item:functors-2} and~\eqref{item:presheaves-2} are equivalent, see Remark~\ref{rem:statements-equiv}.
We now prove~\eqref{item:presheaves-2}.
Set $\cal{A} = \ker(p)$.
Recall from Remark~\ref{rem:calculus} that $p$ admits a right calculus of fractions with respect to $\cal{C}_{\pr}^{\cal{A}}$, so that for every $x,y \in \cal{C}$ the induced map
\[ \colim_{[x' \epi x] \in \cal{A} \rtimes x}\Map(x',y) \to \Map(p(x),p(y)) \]
is an equivalence of spaces, where $\cal{A} \rtimes x \subseteq \cal{C}_{/x}$ is the full subcategory spanned by the projections $x'\epi x$ with fibre in $\cal{A}$. This means that for every map $f\colon p(x) \to p(y)$ we have $\colim_{x' \epi x \in \cal{A} \rtimes x}\Map_f(x',y) \simeq \ast$, where $\Map_f(x',y)$ is the fibre of $\Map(x',y) \to \Map(p(x),p(y))$ over $f$. Viewing the last colimit as the comma category of the composed functor
\[ F\colon  \cal{A}\rtimes x \to \cal{C} \times_{\cal{B}}\{p(x)\} \to \cal{C}_{p(x)/} \]
over $(x,f)$, we conclude that $F$ is coinitial, and so for every $\phi\colon \cal{C}^{\op} \to \E$ we have
\[ (p_!\phi)(p(x)) = \colim_{[x' \epi x] \in \cal{A} \rtimes x}\phi(x) .\]
Noting that $\pst(\cal{C}) \to \pst(\cal{B})$ is again a Keller projection with kernel $\pst(\cal{A})$ by Corollary~\ref{cor:heartwise-keller}, 
to prove the claim it will now suffice to show that for every $x \in \cal{C}$ the induced functor
\[ \cal{A} \rtimes x \to \pst(\cal{A}) \rtimes x \]
is coinitial. Shifting exact sequences inside the stable envelop of $\cal{C}$ we may identify this functor with the inclusion
\[ \cal{A}\times_{\st(\cal{C})} \st(\cal{C})_{\Sigma^{-1}x/} \to \pst(\cal{A})\times_{\st(\cal{C})}\st(\cal{C})_{\Sigma^{-1} x/}, \]
whose comma category over a given $g\colon \Sigma^{-1} x \to A$ with $A \in \pst(\cal{A})$ is the category of factorizations $\cal{A}_{\Sigma^{-1}x//A}$ of $g$ through an object of $\cal{A}$. This category of factorizations is indeed weakly contractible, as can be obtained by applying Lemma~\ref{catoffact} to the composed right-special inclusion $\cal{A} \to \cal{C} \to \st(\cal{C})_{[-1,0]}$ (see Example~\ref{ex:resolving-is-special}).
\end{proof}

\begin{remark}
As in Remark~\ref{rem:beck-chevalley}, passing to opposites we can deduce from Corollary~\ref{cor:beck-chevalley-2} that if $p\colon \cal{C} \to \cal{B}$ is a Keller projection then the square
\[
\begin{tikzcd}[column sep = 50pt]
\cal{C} \ar[r,"p"]\ar[d,"j_{\cal{C}}"'] & \cal{B} \ar[d,"j_{\cal{B}}"] \\
\st(\cal{C})_{[-\infty,0]} \ar[r,"\st(p)_{[-\infty,0]}"] & \st(\cal{B})_{[-\infty,0]} 
\end{tikzcd}
\]
satisfies the Beck-Chevalley properties of Corollary~\ref{cor:beck-chevalley-2}, only that~\eqref{item:functors-2} holds for presheaves and~\eqref{item:presheaves-2} for functors.
\end{remark}

\subsection{Reedy functor categories}\label{sec:reedy}

In this section we show that Keller sequences are stable under taking Reedy functor categories.

Let $\cal{I}$ be a finite poset and let $\cal{E}$ be an exact $\infty$-category. A natural transformation $\vphi \Rightarrow \psi$ of functors $\vphi,\psi\colon \cal{I} \to \cal{E}$ is a Reedy cofibration if for every $i \in \cal{I}$ the square
\[
\begin{tikzcd}
\colim_{j < i} \vphi(j) \ar[r]\ar[d] & \colim_{j < i} \psi(j) \ar[d]\\
\vphi(j) \ar[r] & \psi(j)
\end{tikzcd}
\]
in $\st(\cal{E})$ has total cofibre in $\cal{E}$. We say that $\vphi$ is Reedy cofibrant if the map $0 \to \vphi$ is a Reedy cofibration. Note that if $\vphi$ is Reedy cofibrant then $\vphi|_{\cal{I}'}$ is Reedy cofibrant for any downward closed subposet $\cal{I}'\subseteq \cal{I}$.
We denote by $\fun^{\reedy}(\cal{I},\cal{E}) \subseteq \fun(\cal{I},\cal{E})$ the full subcategory spanned by the Reedy cofibrant objects. This inclusion is closed under extensions and we will consider $\fun^{\reedy}(\cal{I},\cal{E})$ as endowed with the induced exact structure.

\begin{lemma}\label{lem:reedy-inc-pr}
In the induced exact structure on $\fun^{\reedy}(\cal{I},\cal{E})$ a map $\vphi \Rightarrow \psi$ is an inclusion if and only if it is a Reedy cofibration, and a projection if and only if it is pointwise a projection in $\cal{E}$.
\end{lemma}
\begin{proof}
We need to show that if 
\[ \psi \Rightarrow \vphi \Rightarrow \phi \]
is a pointwise exact sequence in $\fun(\cal{I},\cal{E})$ with $\vphi$ Reedy cofibrant then the following holds:
\begin{enumerate}
\item
If $\phi$ is Reedy cofibrant then $\psi$ is Reedy cofibrant.
\item
If $\psi$ is Reedy cofibrant then $\phi$ is Reedy cofibrant if and only if $\psi \Rightarrow \vphi$ is a Reedy cofibration.
\end{enumerate}
For $i \in \cal{I}$ consider the commutative diagram
\[
\begin{tikzcd}
\colim_{j < i}\psi(j) \ar[d,hookrightarrow]\ar[r,hookrightarrow] & \colim_{j < i}\vphi(j) \ar[r]\ar[d,hookrightarrow] & \colim_{j < i}\phi(j)\ar[d,hookrightarrow] \\
\psi(i)\ar[r,hookrightarrow]\ar[d] & \vphi(i) \ar[r]\ar[d] & \phi(i)\ar[d] \\
A \ar[r] & B \ar[r] & C
\end{tikzcd}
\]
in $\st(\cal{E})$ where the last row is obtained by taking vertical cofibres. In particular, all three rows are exact, and since we assume that $\vphi$ is Reedy cofibrant we have $B \in \cal{E}$. Now if $\phi$ is also Reedy cofibrant then $C$ is also in $\cal{E}$, in which case $A \in \st(\cal{E})_{\leq 0}$. But $A$ also lies in $\st(\cal{E})_{\geq 0}$ (since it is by definition the cofibre of a map in $\st(\cal{E})_{\geq 0}$) and so $A \in \cal{E}$. This shows the first claim. For the second claim, assume that $\psi$ is Reedy cofibrant, so that $A$ is in $\cal{E}$. The second claim can then be restated as saying that $A \to B$ is an inclusion in $\cal{E}$ if and only if $C$ is in $\cal{E}$, which holds since $\cal{E} \hookrightarrow \st(\cal{E})$ preserves and detects exact sequences.
\end{proof}

\begin{lemma}\label{lem:colim}
Let $\vphi\colon \cal{I} \to \cal{E}$ be a Reedy cofibrant diagram. Then $\colim_{\cal{I}}\vphi \in \st(\cal{E})$ belongs to $\cal{E}$. As a result, the functor $\colim$ factors through an exact functor
\[ \colim\colon \fun^{\reedy}(\cal{I},\cal{E}) \to \cal{E} .\]
\end{lemma}
\begin{proof}
We prove by induction on the size of $\cal{I}$. Clearly the claim holds if $\cal{I}$ is empty. Now suppose that $\cal{I}$ is non-empty and that the claim has been proven for all posets of strictly smaller size. Since $\cal{I}$ is finite and non-empty there exists an $i \in \cal{I}$ which is maximal in the sense that there does not exist a $j \in \cal{I}$ strictly bigger than $i$. Then the complement $\cal{I} \setminus \{i\}$ is downward closed and $\cal{I}$ is the union of its two downward closed full subposets $\cal{I} \setminus \{i\}$ and $\cal{I}_{\leq i} = \{j \in \cal{I}\;|\; j \leq i\}$, whose intersection is $\cal{I}_{< i} = \{j \in \cal{I}\;|\; j < i\}$. Passing to nerves we then obtain the nerve of $\cal{I}$ is the pushout of the nerve of $\cal{I}_{j \neq i}$ and the nerve of $\cal{I}_{\leq i}$ along the nerve of $\cal{I}_{< i}$, which implies that the colimit of $\vphi$ over $\cal{I}$ fits in a pushout square 
\[
\begin{tikzcd}
\colim_{j < i}\vphi(i)\ar[r]\ar[d] & \colim_{j \leq i}\vphi(j)\ar[d] \ar[r,equal] & \vphi(i)  \\
\colim_{j \neq i}\vphi(j) \ar[r] & \colim_{j}\vphi(j). &
\end{tikzcd}
\]
Now $\cal{I}\setminus \{i\}$ and $\cal{I}_{< i}$ are finite posets of strictly smaller size than $\cal{I}$. By the induction hypothesis the terms in the left column are hence in $\cal{E}$. At the same time, since $\vphi$ is a Reedy cofibration the cofibre of the top horizontal map is in $\cal{E}$, and so the same holds for the cofibre of the bottom horizontal map. We conclude that $\colim_{j}\vphi(j) \in \cal{E}$, as desired. For the last claim, note that $\colim$ is induced by an exact functor among stable $\infty$-categories 
\[ \fun(\cal{I},\st(\cal{E})) \to \st(\cal{E})\] 
and so once we know that it sends the extension-closed full subcategory $\fun^{\reedy}(\cal{I},\cal{E})$ on the left to the extension-closed full subcategory $\cal{E}$ on the right the induced functor
\[ \fun^{\reedy}(\cal{I},\cal{E}) \to \cal{E} \]
is automatically exact with respect to the induced exact structures on both sides.
\end{proof}

We next show that the inclusion $\fun^{\reedy}(\cal{I},\cal{E})  \subseteq \fun(\cal{I},\st(\cal{E}))$ exhibits its target as the stable envelope of the source. To see this, it suffices by~\cite{Sau23} to exhibit a bounded heart structure on $\fun(\cal{I},\st(\cal{E}))$ whose heart is exactly $\fun^{\reedy}(\cal{I},\cal{E})$. The full subcategories
\[ \fun(\cal{I},\st(\cal{E}))_{\geq m} \subseteq \fun(\cal{I},\st(\cal{E})) \supseteq \fun(\cal{I},\st(\cal{E}))_{\leq n} \]
are defined as follows: a functor $\vphi\colon \cal{I} \to \st(\cal{E})$ belongs to $\fun(\cal{I},\st(\cal{E}))_{\geq m}$ exactly when $\vphi(i)$ belongs to $\st(\cal{E})_{\geq m}$ for every $i$, and belongs to $\fun(\cal{I},\st(\cal{E}))_{\leq n}$ exactly when the cofibre of $\colim_{j < i} \vphi(j) \to \vphi(i)$ belongs to $\st(\cal{E})_{\leq n}$ for every $i \in \cal{I}$. 

\begin{lemma}\label{lem:heart-reedy}
The full subcategories $\fun(\cal{I},\st(\cal{E}))_{\geq 0}$ and $\fun(\cal{I},\st(\cal{E}))_{\leq 0}$ defined above form a heart structure on $\fun(\cal{I},\cal{E})$. If $\cal{E}$ is weakly idempotent complete then the heart of this structure coincides with $\fun^{\reedy}(\cal{I},\cal{E})$.
\end{lemma}

\begin{corollary}\label{cor:envelop-reedy}
Let $\cal{E}$ be a weakly idempotent complete exact $\infty$-category. Then the inclusion $\fun^{\reedy}(\cal{I},\cal{E})  \subseteq \fun(\cal{I},\st(\cal{E}))$ exhibits its target as the stable envelope of the source.
\end{corollary}

\begin{proof}[Proof of Lemma~\ref{lem:heart-reedy}]
It is clear that $\fun(\cal{I},\st(\cal{E}))_{\geq 0}$ is closed under finite colimits, $\fun(\cal{I},\st(\cal{E}))_{\leq 0}$ under finite limits , and both are closed under extensions, simply because $\st(\cal{E})_{\geq 0}$ and $\st(\cal{E})_{\leq 0}$ have the analogous properties (where we note that finite limits commute with finite colimits in $\st(\cal{E})$ by stability). It is also clear from the definitions that $\fun(\cal{I},\st(\cal{E}))_{\geq 0} \cap \fun(\cal{I},\st(\cal{E}))_{\leq 0}$ coincides with Reedy cofibrant diagram to $\st(\cal{E})^{\heartsuit}$, and hence to $\cal{E}$ when $\cal{E}$ is weakly idempotent complete. 

It is left to show that there is a sufficient supply of heart decompositions.
Let $\vphi\colon \cal{I} \to \cal{E}$ be some diagram. We wish to construct an exact sequence $\psi^- \Rightarrow \vphi \Rightarrow \psi^+$ with $\psi^- \in \fun(\cal{I},\st(\cal{E}))_{\leq 0}$ and $\psi^+ \in \fun(\cal{I},\st(\cal{E}))_{\geq 1}$. We construct $\psi^-$ and $\pst^+$ by induction on the poset $\cal{I}$. Let $i \in \cal{I}$. Suppose that $i \in \cal{I}$ is such that we have already constructed an exact
\[ \psi^- \Rightarrow \vphi|_{\cal{I}_{<i}} \Rightarrow \psi^+ \]
with $\psi^- \in \fun(\cal{I}_{< i},\st(\cal{E}))_{\leq 0}$ and $\psi^+ \in \fun(\cal{I}_{< i},\st(\cal{E}))_{\geq 1}$, where $\cal{I}_{<i} = \{j \in \cal{I} \;|\; j < i\}$ (note that there is always at least one element in $\cal{I}$ for which $\cal{I}_{<i} = \emptyset$, so the induction can always start somewhere). 
Note that this implies that $\colim_{j < i}\psi^-(j) \in \st(\cal{E})_{\leq 0}$ and $\colim_{j < i}\psi^+(j) \in \st(\cal{E})_{\geq 1}$. In particular we have maps
\[
\begin{tikzcd}
\colim_{j < i}\psi^-(j) \ar[r] & \colim_{j < i}\vphi(j) \ar[r]\ar[d] & \colim_{j < i}\psi^+(j)  \\
& \vphi(i) &
\end{tikzcd}
\]
such that the horizontal row is a heart decomposition for $\colim_{j < i}\vphi(j)$. 
Now consider the pushout $X := \vphi(i) \displaystyle\mathop{\coprod}_{\colim_{j < i}\vphi(j)} \colim_{j < i}\psi^+(j)$ and let
\[ X^- \to X \to X^+ \]
be a heart decomposition with $X^- \in \st(\cal{E})_{\leq 0}$ and $X^+ \in \st(\cal{E})_{\geq 1}$, and set $Y^-:=\fib[\vphi(i) \to X^+] $. 
Consider the resulting diagram
\[
\begin{tikzcd}
\colim_{j < i}\psi^-(j) \ar[r]\ar[d] & \colim_{j < i}\vphi(j) \ar[r]\ar[d] & \colim_{j < i}\psi^+(j)  \ar[d] \\
Y^- \ar[r] & \vphi(i)\ar[r] & X^+ 
\end{tikzcd}
\]
with exact rows. By construction, the total cofibre of the right square is $\cofib[X \to X^+] = \Sigma X^-$. It follows that the cofibre of the left vertical map is equivalent to $X^-$, and is hence in $\st(\cal{E})_{\leq 0}$.
The right hand side then encodes an extension of $\psi^+$ to $\fun(\cal{I}_{\leq i},\st(\cal{E}))_{\geq 1}$ with $\psi^+(i) = X^+$, and the left hand side encodes an extension of $\psi^-$ to $\fun(\cal{I}_{\leq i},\st(\cal{E}))_{\leq 0}$. We have thus extended our desired heart decomposition from $\cal{I}_{< i}$ to $\cal{I}_{\leq i}$. Since $\cal{I}$ is a finite poset this process terminates after finitely many steps, yielding the desired heart decomposition on all of $\cal{I}$.
\end{proof}

\begin{proposition}\label{prop:reedy-keller}
Let $\cal{I}$ be a finite poset and
\[ \cal{A} \to \cal{C} \to \cal{B} \]
a Keller sequence. Then the induced sequence
\[ \fun^{\reedy}(\cal{I},\cal{A}) \to \fun^{\reedy}(\cal{I},\cal{C)} \to \fun^{\reedy}(\cal{I},\cal{B}) \]
is again Keller.
\end{proposition}
\begin{proof}
We first show that the inclusion $\fun^{\reedy}(\cal{I},\cal{A}) \to \fun^{\reedy}(\cal{I},\cal{C)}$ is left-special. Let $\vphi\colon \cal{I} \to \cal{C}$ be a Reedy cofibrant diagram and $\vphi\Rightarrow \psi$ a pointwise projection with $\psi$ a Reedy cofibrant diagram taking values in $\cal{A}$. Since $\cal{A} \to \cal{C}$ is left-special there exists, for every $i \in \cal{I}$, an object $b_i$ equipped with a map $b_i \to \vphi(i)$ such that the composite $b_i \to \vphi(i) \to \psi(i)$ is an inclusion in $\cal{A}$. Define $\phi\colon \cal{I} \to \cal{A}$ by the formula
\[ \phi(i) = \oplus_{j < i} a_j .\]
Note that this formula indeed defines a functor, since the association $i \mapsto S_i := \{j \in \cal{I} \;|\; j < i\}$ is a functor from $\cal{P}$ to sets. Let us now view this functor to $\infty$-groupoids which happens to take values in discrete ones. We then observe that for every $j \in\cal{I}_{< i}$ the poset $\cal{J}_j := \{k \in \cal{I}_{< i} \;|\; j \in S_k\}$ contains $j$ as an initial object, and is hence weakly contractible. We conclude that for every $i \in \cal{I}$ the maps $S_j \to S_i \setminus \{i\}$ for $j \in \cal{I}_{< i}$ exhibit $S_i \setminus \{i\}$ as the colimit of the functor $\cal{I}_{< i} \to \Sps$ sending $j$ to $S_j$. Consequently, for every $i \in \cal{I}$ the induced map $\colim_{j < i}\phi(j) \to \phi(i)$ is equivalent to the map $\oplus_{j < i} b_i \to \oplus_{j \leq i} b_i$, which is an inclusion with cofibre $b_i$. We conclude that $\phi$ is Reedy cofibrant. The composite $\phi \Rightarrow \vphi \Rightarrow \psi$ is then a pointwise projection by construction and hence a projection by Lemma~\ref{lem:reedy-inc-pr}.

We now show that the inclusion $\fun^{\reedy}(\cal{I},\cal{A}) \to \fun^{\reedy}(\cal{I},\cal{C)}$ is right-special. Let $\vphi\colon \cal{I} \to \cal{C}$ be a Reedy cofibrant diagram and $\psi\Rightarrow \vphi$ a Reedy cofibration with $\psi$ a Reedy cofibrant diagram taking values in $\cal{A}$. We construct a diagram $\phi\colon \cal{I} \to \cal{A}$ and a map $\vphi\Rightarrow \phi$ such that the composite $\psi \Rightarrow \vphi \Rightarrow \phi$ is a Reedy cofibration. As in the proof of Lemma~\ref{lem:heart-reedy} we construct $\phi$ by induction on the poset $\cal{I}$. In particular, given an $i \in \cal{I}$ such that we have already constructed $\phi\colon \cal{I}_{<i} \to \cal{A}$ with the desired property with respect to $\vphi|_{I_{<i}}$, we show how to extend $\phi$ from $\cal{I}_{<i}$ to $\cal{I}_{\leq i}$. Consider the diagram
\[
\begin{tikzcd}
\colim_{j < i}\psi(j) \ar[dr, hookrightarrow]\ar[dd,hookrightarrow]\ar[rr,hookrightarrow] && \colim_{j < i}\vphi(j) \ar[dd, dashed,hookrightarrow]\ar[dr] & \\
& \colim_{j < i}\phi(j)\ar[rr,equal]\ar[dd,hookrightarrow] && \colim_{j < i}\phi(j)\ar[dd,hookrightarrow] \\
\psi(i)\ar[rr,hookrightarrow,dashed]\ar[dr,hookrightarrow] && \vphi(i)\ar[dr] &  \\
& a \ar[rr] && c
\end{tikzcd}
\]
where the map $a \to c$ is obtained by forming pushouts in the left and right faces. Here, all terms on the left face are in $\cal{A}$, all terms on the right face are in $\cal{C}$, and all indicated maps are inclusions, as can be deduced from Lemma~\ref{lem:colim} and the assumptions on $\vphi,\psi$ and $\phi$.
In addition, since the front-top horizontal map is the identity the $\cofib[a \to c] \in \st(\cal{C})$ coincides with the total cofibre of the back face, which lies in $\cal{C}$ since $\psi \Rightarrow \vphi$ is a Reedy cofibration. We conclude that $a \to c$ is an inclusion.
Now since $\cal{A} \to \cal{C}$ is right-special there exists a map $c \to b$ in $\cal{C}$ such that the composite $a \to c \to b$ is an inclusion. The rightmost vertical map in the resulting diagram
\[
\begin{tikzcd}
\colim_{j < i}\psi(j) \ar[d,hookrightarrow]\ar[r,hookrightarrow] & \colim_{j < i}\vphi(j) \ar[r]\ar[d,hookrightarrow] & \colim_{j < i}\phi(j)\ar[d,hookrightarrow] \\
\psi(i)\ar[r,hookrightarrow] & \vphi(i) \ar[r] & b
\end{tikzcd}
\]
then encodes an extension $\tilde\phi$ of $\phi$ from $\cal{I}_{< i}$ to $\cal{I}_{\leq i}$ with $\phi(i) = b$, and $\tilde\phi$ is Reedy cofibrant since the rightmost vertical map is an inclusion. In addition, the total cofibre of the external rectangle is equivalent to $\cofib[a \hookrightarrow b]$, which lies in $\cal{A}$, and so the resulting composite $\psi|_{\cal{I}_{\leq i}} \Rightarrow \vphi|_{\cal{I}_{\leq i}} \Rightarrow \tilde\phi$ is a Reedy cofibration.

We have thus established that $\fun^{\reedy}(\cal{I},\cal{A}) \to \fun^{\reedy}(\cal{I},\cal{C)}$ is left- and right-special. It is also closed under retracts, since $\cal{A}$ is closed under retracts in $\cal{C}$, and hence a Keller inclusion. By Proposition~\ref{prop:verdier} it will now suffice to show that 
\[ \fun^{\reedy}(\cal{I},\cal{A}) \to \fun^{\reedy}(\cal{I},\cal{C)} \to \fun^{\reedy}(\cal{I},\cal{B}) \]
is a cofibre sequence. Now by Corollary~\ref{cor:envelop-reedy} the induced sequence on stable envelopes is given by
\[ \fun(\cal{I},\st(\cal{A})) \to \fun(\cal{I},\st(\cal{C})) \to \fun(\cal{I},\st(\cal{B})) ,\]
which is a Verdier sequence by Corollary~\ref{st} and \cite[Proposition 6.5.10]{9-authors-I}, and in particular a cofibre sequence. Since forming stable envelopes is colimit preserving it follows that the induced exact functor from $\cofib[\fun^{\reedy}(\cal{I},\cal{A}) \to \fun^{\reedy}(\cal{I},\cal{C)}]$ to $\fun^{\reedy}(\cal{I},\cal{B})$ is an equivalence on stable envelopes, and hence in particular fully faithful. To finish the proof it will hence suffice to show that 
\[ \fun^{\reedy}(\cal{I},\cal{C)} \to \fun^{\reedy}(\cal{I},\cal{B}) \]
is essentially surjective. Given a Reedy cofibrant diagram in $\vphi\colon \cal{I} \to \cal{B}$ we may construct a $\phi\colon \cal{I} \to \cal{C}$ such that $p(\phi) \simeq \vphi$ by induction on $\cal{I}$ as above. In particular, given a $\phi$ constructed on $\cal{I}_{<i}$ for some $i \in \cal{I}$, extending $\phi$ to $\cal{I}_{\leq i}$ requires solving the following lifting problem: given an object $c \in \cal{C}$ and an inclusion $j\colon p(c) \mono b$ in $\cal{B}$, we need to find an inclusion $\tilde{j}\colon c \mono \tilde{b}$ such that $p(\tilde{j}) = j$. Indeed, this lifting property holds for Keller projections, see Remark~\ref{rem:calculus}.
\end{proof}

\section{Algebraic K-theory of exact \texorpdfstring{$\infty$-}{}categories}\label{sec:K-theory}

In this section we consider the algebraic K-theory functor 
\[ \K\colon \exact \to \Sp .\] 
It can be defined by either by Waldhausen's S-construction, or by Quillen's Q-construction, adapted to exact $\infty$-categories by Barwick~\cite{Bar13}. In the latter formalism one first constructs the associated span $\infty$-category $\Span^{\dagger}(\cal{C}) := \Span(\cal{C},\cal{C}^{\inc},\cal{C}^{\pr})$ whose objects are the objects of $\cal{C}$ and whose morphisms are spans of the form
\[ 
\begin{tikzcd}
& w\ar[dl,->>,"p"']\ar[dr,hookrightarrow,"i"] & \\
x && y
\end{tikzcd}
\]
such that $p$ is a projection and $i$ an inclusion in $\cal{C}$ (which are sometimes referred to as \emph{admissible spans}). One can then define $\K(\cal{C})$ as the connective spectrum whose underlying $\Einf$-group valued functor $\cal{K} = \Omega^{\infty}\K$ is given by the formula
\[ \cal{K}(\cal{C}) = \Omega|\Span^{\dagger}(\cal{C})| \]
where $|-|\colon \cat \to \Sps$ denotes the geometric realization (or $\infty$-groupoid completion) functor, and the $\Einf$-structure comes from the symmetric monoidal structure on $\Span^{\dagger}(\cal{C})$ induced by direct sums in $\cal{C}$.

\begin{remark}\label{rem:opposite}
Let $\cal{C}$ be an exact $\infty$-category and equip its opposite $\cal{C}^{\op}$ with the opposite exact structure, in which the new inclusions are the opposites of the old projections, and similarly for the new projections. We claim that there is a natural equivalence
\[ \Span^{\dagger}(\cal{C}) \simeq \Span^{\dagger}(\cal{C}^{\op}) .\]
Indeed, the (admissible) spans in $\cal{C}^{\op}$ correspond to cospans in $\cal{C}$ of the form
\[ 
\begin{tikzcd}
x \ar[dr,hookrightarrow] && y\ar[dl,->>] \\
& w. &
\end{tikzcd}
\]
But to any such cospan we can take its fibre product to obtain an admissible span in $\cal{C}$ from $x$ to $y$, and this process can be reversed by the taking the pushout of an admissible span. In particular, this yields a natural equivalence $\K(\cal{C}) \simeq \K(\cal{C}^{\op})$.
\end{remark}

\subsection{Additivity for K-theory}\label{sec:additivity}

Recall from~\cite{SW25} that a pair
\[ j\colon \cal{A} \hookrightarrow \cal{C} \hookleftarrow \cal{B} \cocolon i \]
of extension-closed full inclusions among exact $\infty$-categories is said to be a \emph{semi-orthogonal decomposition} of $\cal{C}$ if the functor
    \[
        E(\cal{A},\cal{C},\cal{B}) \to \cal{C}, \quad (a \mono x \epi b) \mapsto x
    \]
    is an equivalence of exact categories, where $E(\cal{A},\cal{C},\cal{B})\subset S_2(\cal{C})$ is the full subcategory spanned by those exact sequences with first term in $\cal{A}$, and last term in $\cal{B}$. By \cite[Proposition 6.2]{SW25}, this is equivalent to the following:
\begin{enumerate}
\item\label{item:exact-right}
$j$ admits an exact right adjoint $q\colon \cal{C} \to \cal{A}$ such that the components of the counit $jq(x) \to x$ are all inclusions and such that $\cal{B} = \ker(q)$.
\item\label{item:exact-left}
$i$ admits an exact left adjoint $p\colon \cal{C} \to \cal{B}$ such that the components of the unit $x \to ip(x)$ are all projections and such that $\cal{A} = \ker(p)$.
\end{enumerate}
In addition, when these equivalent conditions hold the counit of $j \dashv q$ and the unit of $p \dashv i$ fit in an essentially unique manner into an exact sequence
\[ jq(x) \mono x \epi ip(x) ,\]
for every $x \in \cal{C}$. For such a semiorthogonal decomposition, we fix the notation $(\cal{C};\cal{A},\cal{B})$ as referring to the following data:
    \[\begin{tikzcd}
        \cal{A} & \cal{C} \ar[l, shift right=.6ex, bend right=2ex, "j"', "\rotatebox{90}{$\vdash$}", hookleftarrow] \ar[l, shift left=.6ex, bend left=2ex, "q"] \ar[r, bend right=2ex, shift right=.6ex, "i"', hookleftarrow] \ar[r, shift left=.6ex, bend left=2ex, "p", "\rotatebox{90}{$\vdash$}"'] & \cal{B}.
    \end{tikzcd}\]

\begin{remark}\label{rem:semiorthogonal-decompositions-detect}
    Suppose given a semiorthogonal decomposition $(\cal{C};\cal{A},\cal{B})$ of exact $\infty$-categories. Then it follows from the exact sequence of functors $jq \mono \id \epi ip$ and extension closedness of $\cal{C}$ in $\st(\cal{C})$ that the functors $p$, $q$ preserve and jointly detect inclusions and projections.
\end{remark}

\begin{example}\label{ex:universal-semi-orthogonal}
The proto-typical example of semi-orthogonal decomposition is
\[ j\colon \cal{E} \hookrightarrow \rS_2(\cal{E}) \hookleftarrow \cal{E} \cocolon i \]
where $\rS_2(\cal{E})$ is the $\rS$-construction in degree 2, which can be described as the $\infty$-category of exact sequences $x \mono y \epi z$ in $\cal{E}$. Here, the functors $j$ and $i$ are given by $j(x) = [x = x \epi 0]$ and $i(z) = [0 \mono z = z]$, and their exact adjoints $q,p$ as above are given by $q(x \mono y \epi z) = x$ and $p(x \mono y \epi z) = z$.
\end{example}

\begin{example}\label{ex:arrow-semi-orthogonal}
A similar but not isomorphic example of a semi-orthogonal decomposition is
\[ j\colon \cal{E} \hookrightarrow \Ar(\cal{E}) \hookleftarrow \cal{E} \cocolon i \]
where $\Ar(\cal{E})$ is the arrow category of $\cal{E}$ with the pointwise exact structure, $j(y) = [0 \to y]$ and $i(x) = [x \to 0]$. The exact adjoints $p,q$ as above are then given by $q(x\to y) =y$ and $p(x \to y) = x$.
\end{example}

\begin{lemma}
Let $(\cal{C};\cal{A},\cal{B})$ be a semi-orthogonal decomposition.
Then the following holds:
\begin{enumerate}
\item\label{item:split-special-inclusion}
$j\colon \cal{A} \to \cal{C}$ is a left-special inclusion and $i\colon \cal{B} \to \cal{C}$ is a right-special inclusion.
\item\label{item:split-special-projection}
$p\colon \cal{C} \to \cal{B}$ is a left-special projection and $q\colon \cal{C} \to \cal{A}$ is a right-special projection. 
\item\label{item:split-sequences}
The sequences
\[ 
\cal{A} \xrightarrow{j} \cal{C} \xrightarrow{p} \cal{B} 
\quad\text{and}\quad
\cal{B} \xrightarrow{i} \cal{C} \xrightarrow{q} \cal{A} 
\]
are both fibre-cofibre sequences in $\exact$.
\end{enumerate}
\end{lemma}
\begin{proof}
For~\eqref{item:split-special-inclusion} this is an instance of Example~\ref{ex:adjoints}. As for~\eqref{item:split-special-projection} let us just give the argument for $p$, the one for $q$ being completely analogous. Suppose given a map $f\colon x \to y$ in $\cal{C}$ such that $p(f)$ is an equivalence in $\cal{B}$. Consider the commutative diagram
\[ 
\begin{tikzcd}
jq(x) \ar[r,hookrightarrow]\ar[d,"jq(f)"'] & x \ar[d,"f"]\ar[r,->>] & ip(x) \ar[d,"ip(f)","\simeq"'] \\
jq(y) \ar[r,hookrightarrow] & y \ar[r,->>] & ip(y)
\end{tikzcd}
\]
in which the rows are exact sequences. Then the left square is a pushout, so that the induced map $jq(y) \oplus x \to y$ is a projection with kernel $jq(x)$. Setting $a := jq(y)$ we then obtain a factorization
\[
\begin{tikzcd}
& x \oplus a \ar[rd, twoheadrightarrow] \\
x \ar[ru, rightarrowtail] \ar[rr, "f"] && y,
\end{tikzcd}
\]
as a summand inclusion with complement in $\ker(p)$ followed by a projection, so that $p$ is a left-special projection as desired. 

We now prove~\eqref{item:split-sequences}. Since $\ker(q)=\cal{B}$ and $\ker(p)=\cal{A}$ we have that the sequences in question are fibre sequences.
We now show that the sequence
\[ \cal{A} \xrightarrow{i} \cal{C} \xrightarrow{q} \cal{B} \]
is also cofibre sequences in $\exact$; the argument for the second sequence is completely analogous. Let $\cal{E} \in \exact$ be a test object. Then the adjunction $p\dashv i$ induces via pre-composition an adjunction
\[ i^* \colon \funex(\cal{C},\cal{E}) \adj \funex(\cal{B},\cal{E})\cocolon p^* \]
on the level of exact functor categories, whose counit is again an equivalence. In particular, $p^*$ is fully-faithful, and an arbitrary exact functor $\vphi\colon \cal{C} \to \cal{E}$ is in the image of $p^*$ if and only if the unit is an equivalence on $\vphi$, that is, if and only if the map $\vphi(x) \to \vphi(ip(x))$ is an equivalence for every $x$. Now since the unit fits in an exact sequence
\[ jq(x) \mono x \epi ip(x) \]
we have that such an exact $\vphi$ inverts the map $x \epi ip(x)$ if and only if it vanishes on $jq(x)$. We conclude that an exact $\vphi\colon \cal{C} \to \cal{E}$ is in the image of $p^*$ if and only if it vanishes on $jq(x)$ for every $x$, which is equivalent to saying that $\vphi$ vanishes on the essential image of $j$. We conclude that $p$ exhibits $\cal{B}$ as the cofibre of $j\colon \cal{A} \to \cal{C}$, as desired.
\end{proof}

Finally, let us also record the following, which is a version of \cite[Lemma 2.6.1]{9-authors-II} in the setting of exact categories: 
\begin{lemma}\label{lem:p-cartesian}
Let $(\cal{C};\cal{A},\cal{B})$ be a semi-orthogonal decomposition.
Then 
\begin{enumerate}
\item
$p\colon \cal{C} \to \cal{B}$ is a cartesian fibration and an edge $x\to y$ in $\cal{C}$ is $p$-cartesian if and only if the associated square
    \[\begin{tikzcd}
        x\ar[r] \ar[d, \head] & y \ar[d, \head] \\
        ip(x) \ar[r] & ip(y)
    \end{tikzcd}\]
is pullback. 
\item
$q\colon \cal{C} \to \cal{A}$ is a cocartesian fibration and an edge $x\to y$ in $\cal{C}$ is $q$-cocartesian if and only if the associated square
    \[\begin{tikzcd}
        jq(x)\ar[r] \ar[d, \hook] & jq(y) \ar[d, \hook] \\
        x \ar[r] & y
    \end{tikzcd}\]
is pushout.
\end{enumerate}
\end{lemma}
\begin{proof}
The fact that arrows with these properties are (co)cartesian follows immediately from the mapping space criterion~\cite[2.4.4.3]{HTT}. It is also clear that there is a sufficient supply of lifts of this form since one can always form the pullback/pushout along the unit/counit since these are projections and inclusions, respectively, by assumption (where we have also used that $p$ and $q$ have fully-faithful adjoints).
\end{proof}

\begin{definition}\label{def:additive}
Let $\cal{E}$ be an additive $\infty$-category. We shall say that a functor
\[ \cal{F}\colon \exact \to \cal{E} \]
is \emph{additive} if it preserves zero objects and for every semi-orthogonal decomposition of exact $\infty$-categories
    \[\begin{tikzcd}
        \cal{A} & \cal{C} \ar[l, shift right=.6ex, bend right=2ex, "j"', "\rotatebox{90}{$\vdash$}", hookleftarrow] \ar[l, shift left=.6ex, bend left=2ex, "q"] \ar[r, bend right=2ex, shift right=.6ex, "i"', hookleftarrow] \ar[r, shift left=.6ex, bend left=2ex, "p", "\rotatebox{90}{$\vdash$}"'] & \cal{B}.
    \end{tikzcd}\]
the maps
\[ \K(\cal{A}) \oplus \K(\cal{B}) \xrightarrow{(\K(i), \K(j))} \K(\cal{C}) \xrightarrow{(\K(p),\K(q))} \K(\cal{A}) \oplus \K(\cal{B}) \]
are equivalences.
\end{definition}

\begin{remark}\label{rem:additive-S}
By Example~\ref{ex:universal-semi-orthogonal} we have that if $\cal{F}\colon \exact \to \cal{E}$ is additive then $\cal{F}(\rS_2(\cal{C})) \simeq \K(\cal{C})\oplus \K(\cal{C})$. More generally, arguing by induction we obtain that
\[ \cal{F}(\rS_n(\cal{C})) \simeq \cal{F}(\cal{C}) \oplus ... \oplus \cal{F}(\cal{C}) \]
is a direct sum of $n$ copies of $\cal{F}(\cal{C})$. Noting that objects in $\rS_n(\C)$ correspond to objects in $\cal{C}$ equipped with a length $n$ filtration and that the equivalence above is induced by passing to the associated graded pieces it follows by standard arguments that if $f\colon \cal{C}' \to \cal{C}$ is an exact functor that admits a length $n$ filtration $0 = f_0 \mono ... \mono f_{n} = f$ in the associated exact functor category then the induced map $\cal{F}(f)\colon \cal{F}(\cal{C}') \to \cal{F}(\cal{C})$ is equivalent to the sum of the induced maps $\cal{F}(\cofib[f_i \to f_{i+1}])$ for $i=0,...,n-1$.
\end{remark}

\begin{proposition}[Additivity]\label{prop:additivity}
The algebraic $\K$-theory functor $\K\colon \exact \to \Sp$ is additive.
\end{proposition}

Additivity for $\K$-theory can be expressed in various forms; it is in fact enough to prove it for Example~\ref{ex:universal-semi-orthogonal}, a statement which itself is sometimes referred to as additivity. It is due to Waldhausen in the context of Waldhausen categories, and extended by Barwick to Waldhausen $\infty$-categories in~\cite{Bar16}. A self-contained proof for stable $\infty$-categories can be found in~\cite[Corollary 4.2]{HLS24}. In a similar spirit, we shall henceforth give a self-contained proof for additivity in the setting of exact $\infty$-categories.

\begin{lemma}\label{lem:lifts-for-span}

Let 
    \[\begin{tikzcd}
        \cal{A} & \cal{C} \ar[l, shift right=.6ex, bend right=2ex, "j"', "\rotatebox{90}{$\vdash$}", hookleftarrow] \ar[l, shift left=.6ex, bend left=2ex, "q"] \ar[r, bend right=2ex, shift right=.6ex, "i"', hookleftarrow] \ar[r, shift left=.6ex, bend left=2ex, "p", "\rotatebox{90}{$\vdash$}"'] & \cal{B}.
    \end{tikzcd}\]
be a semi-orthogonal decomposition.
Then a $p$-cartesian inclusion in $\cal{C}$ is also a $\Span^{\dagger}(p)$-cartesian arrow in $\Span^{\dagger}(\cal{C})$ when viewed as a purely forward span and a $p$-cartesian projection is also a $\Span^{\dagger}(p)$-cocartesian arrow in $\Span^{\dagger}(\cal{C})$ when viewed as a purely backward span.

Dually, any $q$-cocartesian inclusion in $\cal{C}$ is also $\Span^{\dagger}(q)$-cocartesian when viewed as a purely forward span and any $q$-cocartesian projection is $\Span^{\dagger}(q)$-cartesian when viewed as a purely backward span.
\end{lemma}

\begin{proof}
The dual variant for $q$ can be obtained from the statement for $p$ by replacing $(\cal{C};\cal{A},\cal{B})$ with $(\cal{C}^{\op};\cal{B}^{\op},\cal{A}^{\op})$ and using Remark~\ref{rem:opposite}. We now prove the statement for $p$.
 
We apply~\cite[Theorem 3.1]{two-variable-fibrations} with $f$ being an equivalence (a case for which the two assumptions of that theorem clearly hold), once to the map of adequate triples $(\cal{C},\cal{C}^{\pr},\cal{C}^{\inc}) \to (\cal{B},\cal{B}^{\pr},\cal{B}^{\inc})$ and once to the map of adequate triples $(\cal{C},\cal{C}^{\inc},\cal{C}^{\pr}) \to (\cal{B},\cal{B}^{\inc},\cal{B}^{\pr})$ (noting that switching the role of ingressive and egressive arrows replaces the associated span $\infty$-category with its opposite). To obtain the desired result it will now suffice to show that any $p$-cartesian projection in $\C$ is also $p^{\pr}$-cartesian and any $p$-cartesian inclusion in $\C$ is also $p^{\inc}$-cartesian, where $p^{\pr}$ and $p^{\inc}$ are the functors induces on the wide subcategories of projections and inclusions, respectively. Concretely, we need to show that in any commutative triangle
\[
\begin{tikzcd}
& y\ar[dr,"g"] & \\
x\ar[ur,"f"]\ar[rr,"h"] && z
\end{tikzcd}
\]
in $\C$ such that $g$ is $p$-cartesian, if $g,h$ and $p(f)$ are projections then so is $f$, and if $g,h$ and $p(f)$ are inclusion then so is $f$. Indeed, this follows from Remark~\ref{rem:semiorthogonal-decompositions-detect} since $q\colon \cal{C} \to \cal{A}$ sends all $p$-cartesian arrows to equivalences (as can be seen by their explicit description appearing in Lemma~\ref{lem:p-cartesian}). 
\end{proof}

Recall that an orthogonal factorisation system \cite[Definition 5.2.8.8]{HTT} on an $\infty$-category $\cal{D}$ is the data of collections $(\cal{D}^l, \cal{D}^r)$ of arrows in $\cal{D}$, each closed under retracts in $\mrm{Ar}(\cal{D})$, and such that any map $x \to y$ factors as $x \epi z \hookrightarrow y$ with $x \epi z$ in $\cal{D}^l$ and $z \hookrightarrow y$ in $\cal{D}^r$, and for each square
\[\begin{tikzcd}
    a \ar[d, \head] \ar[r] & x \ar[d, \hook] \\
    b \ar[r] & y,
\end{tikzcd}\]
the mapping space $\Map_{\cal{D}_{a//y}}(b, x)$ is contractible. Here, the notation is inspired by the example of sets (or abelian groups), which admits an orthogonal factorization system in which the left class consists of the surjective maps and the right class of injective maps.
Another example, which is the one we will use in the present context, is the collection of purely forward and purely backward spans, which give rise to an orthogonal factorisation system on the span $\infty$-category associated to a given adequate triple by \cite[Proposition 4.9]{two-variable-fibrations}. 

\begin{remark}\label{rem:factorisation-system}
    Suppose $\cal{D}$ has an orthogonal factorisation system.
    Observe that for any commutative square
    \[
    \begin{tikzcd}
    a \ar[d,"g"']\ar[r,"i"] & x\ar[d,"f"] \\
    b \ar[r,"j"] & y
    \end{tikzcd}
    \]
    the mapping space $\Map_{\cal{D}_{a//y}}(b, x)$ is obtained as a pullback
    \[\begin{tikzcd}
    \Map_{\cal{D}_{a//y}}(b, x) \ar[r] \ar[d] & \Map_\cal{D}(b, x) \ar[d] \\
    * \ar[r, "{(i, j)}"] & \Map_\cal{D}(a,x) \times_{\Map_\cal{D}(a,y)}\Map_\cal{D}(b,y).
\end{tikzcd}\]
It then follows that for maps $g\colon a \epi b$ and $f\colon x \mono y$ in $\cal{D}^l$ and $\cal{D}^r$ respectively, all the total fibres of the square
\[\begin{tikzcd}
    \Map_\cal{D}(b,x) \ar[r] \ar[d] & \Map_\cal{D}(a,x) \ar[d] \\
    \Map_\cal{D}(b,y) \ar[r] & \Map_\cal{D}(a,y)
\end{tikzcd}\]
are contractible, and so the natural map $\Map(b,x) \to \Map_{\cal{D}^{\Delta^1}}(g, f)$ is an equivalence.
\end{remark}

\begin{proposition}\label{prop:product-criteria}
    Let $\cal{D}$ be an $\infty$-category equipped with an orthogonal factorisation system $(\cal{D}^l,\cal{D}^r)$, such that there is some object $0 \in \cal{D}$ which is initial for maps in $\cal{D}^r$.
    Let $p\colon\cal{E} \to \cal{D}$ be a functor satisfying the following:
    \begin{enumerate}
        \item Maps in $\cal{D}^l$ admit $p$-cocartesian lifts given any lift of their domain and maps in $\cal{D}^r$ admit $p$-cartesian lifts for any lift of their target.
        \item The fibre inclusion $\cal{F} := p^{-1}(0) \to \cal{E}$ admits a retraction $r\colon \cal{E} \to \cal{F}$. 
    \end{enumerate}
    Then the functor 
    \[ (r,p)\colon \cal{E} \to  \cal{F} \times \cal{D} \]
    is a weak homotopy equivalence.
\end{proposition}
\begin{proof}
    Set $\cal{E}_{/\cal{D}} := \cal{E} \times_{\cal{D}^{\Delta^{\{0\}}}} \cal{D}^{\Delta^1}$, so that the composite
    \[ q\colon \cal{E}_{/\cal{D}} \to \cal{D}^{\Delta^1} \xrightarrow{\ev_{1}} \cal{D} \]
    is a cocartesian fibration whose fibre over $x \in \cal{D}$ is the comma category $\cal{E}_{/x}$. Note that the inclusion $\cal{E} \to \cal{E}_{/\cal{D}}$ sending $a$ to $(a,\id\colon p(a)\to p(a))$ is co-initial and in particular a weak homotopy equivalence, and so it will suffice to show that the functor
    \[ (r',q)\colon \cal{E}_{/\cal{D}} \to \cal{F} \times \cal{D} \]
    is a weak homotopy equivalence, where $r'\colon \cal{E}_{/\cal{D}} \to \cal{F}$ is given by $r'(a,p(a) \to x) = r(a)$. Viewing $(r',q)$ as a map among cocartesian fibrations over $\cal{D}$ (where on the right hand side the map to $\cal{D}$ is the projection on the second factor) it preserves cocartesian arrows over $\cal{D}$ (since the $q$-cocartesian arrows $(a,p(a) \to x) \to (a',p(a') \to x')$ are exactly those for which the map $a \to a'$ is an equivalence). To show that $(r',q)$ is a weak homotopy equivalence it will hence suffice to show that the induced map
    \[ r'_x\colon \cal{E}_{/x} \to \cal{F} \]
    on fibres over $x \in \cal{D}$ is a weak homotopy equivalence for every $x$.

    Let $\cal{E}^r := \cal{E} \times_{\cal{D}} \cal{D}^r$ be the wide subcategory of $\cal{E}$ spanned by those morphisms whose image in $\cal{D}$ lies in $\cal{D}^r$, and let $\cal{E}^r_{/x}$ be the comma category of $\cal{E}^r \to \cal{D}^r$ over $x$. Then the induced functor
    \[ \cal{E}^r_{/x} \to \cal{E}_{/x} \]
    is fully-faithful: indeed, by uniqueness of factorizations arrows in $\cal{D}^r$ have the right cancellation property in the sense that if $f,g$ are two composable arrows in $\cal{D}$ such that $g$ and $g\circ f$ are in $\cal{D}^r$, then $f$ is in $\cal{D}^r$.     
    Now, given $(a, f\colon p(a) \to x) \in \cal{E}_{/x}$, we may factor $f$ as $p(a) \stackrel{g}\epi y \stackrel{i}\hookrightarrow x$, and choose a $p$-cocartesian lift $a \to \tilde{y}$ of $g$.
    We claim that the collection of maps 
    \[ 
    (a, p(a)\xrightarrow{f} x) \to (\tilde{y},  p(\tilde{y})=y \stackrel{i}{\hookrightarrow} x)
    \] 
    assemble into the unit of an adjunction $\cal{E}_{/x} \adj \cal{E}^r_{/x}$. Indeed, given $(a', f'\colon p(a') \mono x) \in \cal{E}^r_{/x}$, the map
    \begin{equation}\label{mapping-spaces}
        \Map_{\cal{E}^r_{/x}}((\tilde{y}, y \hookrightarrow x), (a', f')) \to \Map_{\cal{E}_{/x}}((a, p(a) \to x), (a', f'))
    \end{equation}
    is induced on pullbacks by the map of cospans
    \[\begin{tikzcd}
    		&& * \ar[rd] \ar[dd, "i"] \\
    		&&& * \ar[dd, "f"] \\
    		\Map_\cal{E}(\tilde{y},a') \ar[rd] \ar[rr] && \Map_\cal{D}(y,x) \ar[rd] \\
    		& \Map_\cal{E}(a, a') \ar[rr] && \Map_\cal{D}(p(a),x),
    \end{tikzcd}\]
    in which the bottom face is a composition of squares
    \[\begin{tikzcd}
    		\Map_\cal{E}(\tilde{y},a') \ar[r] \ar[d] & \Map_\cal{D}(y, p(a')) \ar[r] \ar[d] & \Map_\cal{D}(y, x) \ar[d] \\
    		\Map_\cal{E}(a,a') \ar[r] & \Map_\cal{D}(p(a),p(a')) \ar[r] & \Map_\cal{D}(p(a), x).
    \end{tikzcd}\]
    We now observe that the left square is cartesian by virtue of the facts that $a \to \tilde{y}$ is $p$-cocartesian and the right square is cartesian by applying Remark \ref{rem:factorisation-system} to $g\colon f(a) \epi y$ and $f'\colon p(a')\mono x$.
    It then follows by pasting that the external rectangle is cartesian and hence the map (\ref{mapping-spaces}) is an equivalence.
 
    Let $\cal{F} \to \cal{E}^r_{/x}$ be the fully-faithful functor sending $b \in \cal{F} = p^{-1}(0)$ to the pair $(b,0 \hookrightarrow x)$, where $0 \hookrightarrow x$ is the essentially unique arrow in $\cal{D}^r$ from $0$ to $x$.     
    Given $(a, f) \in \cal{E}^r_{/x}$, write $b \to a$ for a cartesian lift of the unique $\cal{D}^r$-map $0 \hookrightarrow p(a)$; then by an analogous argument, the collection of maps $(b, 0 \hookrightarrow x) \to (a, p(a) \hookrightarrow x)$ assemble into the counit of an adjunction $\cal{F} \adj \cal{E}^r_{/x}$.

    Since adjunctions induce inverse equivalences on geometric realizations the inclusions $\cal{F} \subseteq \cal{E}^r_{/x} \subseteq \cal{E}_{/x}$ are both weak homotopy equivalences. Since $r'_x\colon \cal{E}_{/x} \to \cal{F}$ restricts to the identity on $\cal{F}$ by construction we conclude that $r'_x$ is indeed a weak homotopy equivalence, as desired.    
\end{proof}

\begin{proof}[Proof of Proposition~\ref{prop:additivity}]
Apply Proposition~\ref{prop:product-criteria} to the induced functor
\[ \Span^{\dagger}(p)\colon \Span^{\dagger}(\cal{C}) \to \Span^{\dagger}(\cal{B}) \]
using the backward-forward unique factorization system on $\Span^{\dagger}(\cal{B})$ of \cite[Proposition 4.9]{two-variable-fibrations}, where we note that the first assumption is satisfied by Lemma~\ref{lem:lifts-for-span} and the second assumption is satisfied since the exact functor $q\colon \cal{C} \to \cal{A}$ induces a retraction $\Span^{\dagger}(q)\colon \Span^{\dagger}(\cal{C}) \to \Span^{\dagger}(\cal{A})$ to the inclusion of the fibre of $\Span^{\dagger}(p)$ over $0 \in \Span^{\dagger}(\cal{B})$.
\end{proof}

\subsection{Localisation for K-theory}\label{sec:k-theory}

\begin{definition}
Let $\cal{E}$ be an additive $\infty$-category with finite limits and $\cal{F}\colon \exact \to \cal{E}$ a functor. We shall say that $\cal{F}$ is \emph{Keller localising} if it is additive (see Definition~\ref{def:additive}) and sends Keller sequences to fibre sequences.
\end{definition}

Our goal in this subsection is to give a self-contained proof of the fact that algebraic $\K$-theory is Keller localising: 

\begin{theorem}\label{thm:localisation}
The algebraic $\K$-theory functor $\K\colon \exact \to \Sp$ is Keller localising.
\end{theorem}

In light of Corollary~\ref{st}, Theorem~\ref{thm:localisation} follows from the $\infty$-categorical version of the Gillet-Waldhausen theorem as established in \cite[Theorem 1.8]{SW25}, together with the localisation theorem for K-theory of stable $\infty$-categories (see, e.g., \cite[Theorem 6.1]{HLS24}). In this section we will show that Theorem~\ref{thm:localisation} can also be deduced directly from the axioms of Keller sequences. In \S\ref{sec:resolution} below we will then use this to give an alternative proof of the $\infty$-categorical Gillet-Waldhausen theorem, therefore generalizing the latter to all finitary Keller localising functors.

For what follows, let us fix a Keller sequence $\cal{A} \to \cal{C} \xrightarrow{p} \cal{B}$ and let us denote by $\cal{C}^{p}$ the wide subcategory of $\cal{C}$ whose morphisms are the morphisms $f$ such that $p(f)$ is an equivalence. We note that $\cal{C}^p$ contains $\cal{C}^{\cal{A}}_{\inc}$ and $\cal{C}^{\cal{A}}_{\pr}$ as wide subcategories. 

\begin{proposition}\label{prop:weak-equivalence}
The maps
\[ \cal{C}^{\cal{A}}_{\inc} \rightarrow \cal{B}^{\simeq} \leftarrow \cal{C}^{\cal{A}}_{\pr} \]
are weak homotopy equivalences.
\end{proposition}

The proof of Proposition~\ref{prop:weak-equivalence} will require a few intermediate results.

\begin{proposition}\label{prop:pass-to-pst}
Let $\cal{A} \to \cal{C} \to \cal{B}$ be a Keller sequence. Then the square of $\infty$-groupoids
\[
\begin{tikzcd}
	{|\cal{C}^{\cal{A}}_\pr|} \ar[r] \ar[d] & \cal{B}^\simeq \ar[d] \\
	{|\pst(\cal{C})^{\pst(\cal{A})}_\pr|} \ar[r] & \pst(\cal{B})^\simeq
\end{tikzcd}
\]
is cartesian.
\end{proposition}
\begin{proof}

We show that the square
\[
\begin{tikzcd}
	{|(\st(\cal{C})_{[0,n-1]})^{\st(\cal{A})_{[0,n-1]}}_\pr|} \ar[r] \ar[d] & \pst(\cal{B})^\simeq_{[0,n-1]}\ar[d] \\
	{|(\st(\cal{C})_{[0,n]})^{\st(\cal{A})_{[0,n]}}_\pr|} \ar[r] & \pst(\cal{B})^\simeq_{[0,n]}
\end{tikzcd}
\]
is cartesian for every $n \geq 1$, thus implying the result upon pasting the squares and passing to the colimit. Equivalently, since geometric realizations commute with base change along maps among $\infty$-groupoids, what we need to show is that the induced functor
\[ (\st(\cal{C})_{[0,n-1]})^{\st(\cal{A})_{[0,n-1]}}_\pr \to (\st(\cal{C})_{[0,n]})^{\st(\cal{A})_{[0,n]}}_\pr \times_{\pst(\cal{B})^\simeq_{[0,n]}} \pst(\cal{B})^\simeq_{[0,n-1]} =: \cal{P} 
\]
is an equivalence on geometric realizations. We claim that this functor is coinitial, hence in particular an equivalence on geometric realizations. For this, let $C$ be an object of $\cal{P}$, which we can view as an object of $\st(\cal{C})_{[0,n]}$ whose image in $\pst(\cal{B})_{[0,n]}$ lies in $\pst(\cal{B})_{[0,n-1]}$, and let $\cal{I}_C$ be the comma category of the above functor over $C$. We need to show that $\cal{I}_{C}$ is weakly contractible. We note that if $D \epi C$ is a projection in $\st(\cal{C})_{[0,n]}$ with domain in $\st(\cal{C})_{[0,n-1]}$ then $\fib(D \epi C)$ also lies in $\st(\cal{C})_{[0,n-1]}$. We hence describe $\cal{I}_C$ as the category whose objects are projections $D \epi C$ in $\st(\cal{C})_{[0,n]}$ with domain in $\st(\cal{C})_{[0,n-1]}$ and fibre in $\st(\cal{A})_{[0,n-1]}$, and whose morphisms are commutative triangles
\[
\begin{tikzcd}
D \ar[dr,->>]\ar[rr,->>] && D' \ar[dl,->>] \\
& C & 
\end{tikzcd}
\]
of projections. We first show that $\cal{I}_{C}$ is not empty. Since $p(C)$ belong to $\pst(\cal{B})_{[0,n-1]}$ we have that the map $C \to 0$ in $\st(\cal{C})_{[0,n]}$ is sent to an inclusion in $\pst(\cal{B})_{[0,n]}$. Now $\st(\cal{C})_{[0,n]} \to \pst(\cal{B})_{[0,n]}$ is a Keller projection by Corollary~\ref{cor:heartwise-keller} and hence inclusions in $\pst(\cal{B})_{[0,n]}$ can be lifted to inclusions in $\st(\cal{C})_{[0,n]}$ with prescribed lift of their domain, see Remark~\ref{rem:calculus}. We conclude that there exists an inclusion $C \mono A$ in $\st(\cal{C})_{[0,n]}$ with $A \in \st(\cal{A})_{[0,n]}$.
Choose a heart decomposition $a \to A \to A'$ with $A' \in \st(\cal{A})_{[1,n]}$ and $a \in \cal{A}$, so that $A \to A'$ is in particular an inclusion in $\st(\cal{A})_{[0,n]}$ (with cofibre $\Sigma a \in \st(\cal{A})_{[1,1]} \subseteq \st(\cal{C})_{[0,n]}$). Finally, consider the composite inclusion
\[ C \mono A \mono A' .\]
Since $A' \in \st(\cal{A})_{[1,n]}$ the cofibre of the inclusion $C \mono A'$ must lie in $\st(\cal{C})_{[1,n]}$ and so this inclusion is also a projection and its fibre $D = \fib[C \to A']$ lies in $\st(\cal{C})_{[0,n-1]}$. Since the cofibre of $D \to C$ lies in $\st(\cal{A})_{[1,n]}$ this map is also a projection with fibre in $\st(\cal{A})_{[0,n-1]})$, and hence determines an object of $\cal{I}_C$. We have thus shown that $\cal{I}_C$ is non-empty. Now we note that $\cal{I}_C$ almost has cartesian products. More precisely, we may embed $\cal{I}_C$ as a wide subcategory of $\cal{J}_C$, where $\cal{J}_C$ is the full subcategory of the comma category $(\st(\cal{C})_{[0,n]})_{/C}$ spanned by the projections $D \epi C$ with fibre in $\st(\cal{A})_{[0,n-1]}$, then $\cal{J}_C$ has cartesian products, given by taking fibre product over $\cal{C}$. This cartesian product restricts to a functorial binary operation $(-)\times_C(-)\colon\cal{I}_C \times \cal{I}_C \to \cal{I}_C$, though this operation is not the cartesian product in $\cal{I}_C$. Nonetheless, the natural projections from $(-)\times_C(-)$ to each of the two factors are pointwise in $\cal{I}_C$, and can hence be viewed as natural transformations of functors $\cal{I}_C \times \cal{I}_C \to \cal{I}_C$. 
Since $\cal{I}_C$ is non empty, this is enough to deduce that $\cal{I}_C$ is weakly contractible using Quillen's classical trick: given an $X \in \cal{I}_C$ the maps $Y \leftarrow X \times Y \rightarrow X$ determine a zig-zag of natural transformations from the identity to a constant functor.
\end{proof}

\begin{lemma}\label{lem:wide}
Let $\cal{E}$ be an $\infty$-category which admits pushouts and $\cal{E}^{\circ}\hookrightarrow \cal{E}$ a wide subcategory satisfying the following properties:
\begin{enumerate}
\item\label{item:cobase-change}
The arrows in $\cal{E}^{\circ}$ are closed under cobase change.
\item\label{item:cancellation}
Given a composable sequence $x \xrightarrow{f} y \xrightarrow{g} z$ in $\cal{E}$ such that $f$ and $g\circ f$ belong to $\cal{E}^{\circ}$, then $g$ belongs to $\cal{E}^{\circ}$.
\item\label{item:supply}
For every map $f\colon x \to y$ in $\cal{E}$ there exists a map $g\colon y \to z$ in $\cal{E}^\circ$ such that $g\circ f$ belongs to $\cal{E}^{\circ}$.
\end{enumerate}
Then the inclusion $\cal{E}^{\circ}\hookrightarrow \cal{E}$ is a weak homotopy equivalence.
\end{lemma}
\begin{proof}
Let $\tilde{\cal{E}} \subseteq \cal{E}^{\Delta^1} \times_{\cal{E}^{\Delta^{\{1\}}}} \cal{E}^{\circ}$ be the full subcategory spanned by the arrows $x \to y$ in $\cal{E}^{\circ}$. In other words, the objects of $\tilde{\cal{E}}$ are the arrows in $\cal{E}^{\circ}$ and the morphisms from $[x \to y]$ to $[x' \to y']$ are the commutative squares
\[
\begin{tikzcd}
x \ar[r]\ar[d] & y \ar[d] \\
x'\ar[r] & y'
\end{tikzcd}
\]
such that $y \to y'$ is in $\cal{E}^{\circ}$. 
We may then factor the map $\cal{E}^{\circ} \to \cal{E}$ as a composite
\[ \cal{E}^{\circ} \xrightarrow{j} \tilde{\cal{E}} \xrightarrow{q} \cal{E} \]
where $j(y) = [\id_y\colon y \to y]$ and $q(x \to y) = x$.
We claim that both $j$ and $q$ are weak homotopy equivalences. For $j$, note that it admits a retraction 
\[ r\colon \tilde{\cal{E}} \to \cal{E}^{\circ} \]
sending $[x \to y]$ to $y$. Furthermore, $r$ is a cocartesian fibration with fibre over $y$ given by the full subcategory of $\cal{E}_{/y}$ given by the arrows in $\cal{E}^{\circ}$ and cocartesian transition functors given by post-composition. Since the fibres of $r$ have terminal objects we see that $r$ is a weak homotopy equivalence, and hence $j$ is a weak homotopy equivalence as well.

We now claim that $q$ is a localisation functor, and so in particular a weak homotopy equivalence.
To see this, we use Rezk's sufficient criterion for localisation, namely, we show that for every $[m] \in \Delta$ the map
\[ q_m\colon \fun^{q}(\Delta^m,\tilde{\cal{E}}) \rightarrow \fun(\Delta^m,\cal{E})^{\simeq} \]
is a weak homotopy equivalence, where on the right we have the core $\infty$-groupoid of the $\infty$-category of functors $\Delta^m \to \cal{E}$ and on the left the wide subcategory of $\fun(\Delta^m,\tilde{\cal{E}})$ whose morphisms are the natural transformations sent to equivalences in $\fun(\Delta^m,\cal{E})$. In particular, the right hand side is an $\infty$-groupoid and hence it suffices to check that the fibres of $q_m$ are weakly contractible. Now given an $\alpha = (x_0 \to ... \to x_m)\in \fun(\Delta^m,\cal{E})$,
the fibre of $q_m$ over $\alpha$ has objects diagrams
\begin{equation}\label{eq:typical-object}
\begin{tikzcd}
x_0\ar[r]\ar[d] & ... \ar[r]\ar[d] & x_m\ar[d]\\
y_0\ar[r] & ... \ar[r] & y_m
\end{tikzcd}
\end{equation}
where all vertical morphisms and all horizontal morphisms in the bottom row belong to $\cal{E}^{\circ}$,
and morphisms natural transformations of such diagrams which are the identity on the top row and pointwise in $\cal{E}^{\circ}$ on the bottom row. To finish the proof we now show that $\cal{I}_{\alpha}$ is non-empty and admits binary coproducts; this implies that $\cal{I}$ is sifted, and in particular weakly contractible. 
Concerning coproducts, we note that by Assumption~\ref{item:cancellation} we may identify $\cal{I}_C$ with a full subcategory of the slice $\infty$-category $\fun(\Delta^m,\cal{E})_{[x_0 \to ... \to x_m]/}$ spanned by diagrams as above. 
Now $\fun(\Delta^m,\cal{E}_{[x_0 \to ... \to x_m]/})$ admits coproducts given by pushouts under $[x_0 \to ... \to x_m]$ and by Assumptions~\ref{item:cobase-change} and~\ref{item:cancellation} we see that 
$\cal{I}_{\alpha}$ is closed under binary coproducts in $\fun(\Delta^m,\cal{E})_{[x_0 \to ... \to x_m]/}$ 
and hence inherits itself binary coproducts. To finish the proof we now show that
$\cal{I}_{\alpha}$ is not empty. For this, we construct the $y_i$'s by induction on $i$. For $i=0$ we set $y_0=x_0$. Now suppose that for $0 \leq i < m$ we have constructed a partial diagram of the form
\begin{equation}
\begin{tikzcd}
x_0\ar[r]\ar[d] & ... \ar[r]\ar[d] & x_i\ar[d]\ar[r] & x_{i+1}\ar[r] & ... \ar[r] & x_m\\
y_0\ar[r] & ... \ar[r] & y_i &&&
\end{tikzcd}
\end{equation}
such that the vertical arrows and the horizontal arrows in the bottom row are in $\cal{E}^{\circ}$. Let $z_{i+1} := y_{i} \coprod_{x_i} x_{i+1}$, so that the map $x_{i+1} \to z_{i+1}$ is in $\cal{E}^{\circ}$ by Assumption~\ref{item:cobase-change}. Then by Assumption~\ref{item:supply} there exists a map $z_{i+1} \to y_{i+1}$ in $\cal{E}^{\circ}$ such that the composite $y_i \to z_{i+1} \to y_{i+1}$ is also in $\cal{E}^{\circ}$. We can then use the $\cal{E}^{\circ}$ maps $y_i \to y_{i+1}$ and $x_{i+1} \to y_{i+1}$ to extend the above diagram from $i$ to $i+1$. Arriving at $i=m$ we have thus constructed an object of $\cal{I}_{\alpha}$, showing that the latter is non-empty, as desired.
\end{proof}

\begin{proposition}\label{prop:remove-pr}
The wide subcategory inclusion 
\[ \iota \colon \pst(\cal{C})^{\pst(\cal{A})}_{\pr} \hookrightarrow \pst(\cal{C})^{\pst(p)}\]
is a weak homotopy equivalence.
\end{proposition}
\begin{proof}
We verify that $\iota$ satisfies the assumptions of Lemma~\ref{lem:wide}. First note that $\pst(\cal{C})$ admits pushouts which are preserved by $\pst(\cal{C}) \to \pst(\cal{B})$, and $\pst(\cal{B})^{\simeq}$ has pushouts (since any $\infty$-groupoid admits colimits indexed by weakly contractible categories), so that $\pst(\cal{C})^{\pst(p)}$ admits pushouts. Next we note that that projections in $\pst(\cal{C})$ are closed under cobase change. Indeed, all maps in $\pst(\cal{C})$ are inclusions and projections are closed under cobase change along inclusions in any exact $\infty$-category. It follows that $\iota$ satisfies Assumption~\eqref{item:cobase-change} holds. In addition, if $C \xrightarrow{f} D \xrightarrow{g} E$ is a composable sequence in $\pst(\cal{C})$ such that $f$ and $g\circ f$ are projections then $g$ is a projection: indeed, the cofibres of $f,g$ and $g\circ f$ fit into an exact sequence
\[ \cofib(f) \to \cofib(g\circ f) \to \cofib(g) \]
and a map in $\pst(\cal{C})$ is a projection if and only if its cofibre is in the full subcategory $\pst(\cal{C})_{[1,\infty]} \subseteq \pst(\cal{C})$ which is closed under cofibres. We conclude that $\iota$ satisfies Assumption~\eqref{item:cancellation}. Finally, suppose that $f\colon C \to D$ is a map in $\pst(\cal{C})$ lying over an equivalence in $\pst(\cal{B})$. 
The map $D \to \cofib[C \to D]$ to the cofibre is a projection in $\pst(\cal{C})$ with target in $\pst(\cal{A})$. Since $\pst(\cal{A}) \to \pst(\cal{C})$ is a Keller inclusion (see Corollary~\ref{cor:heartwise-keller}) there exists an $A \in \pst(\cal{A})$ and a map $A \to D$ such that the composite $A \to D \to \cofib[C \to D]$ is a projection in $\pst(\cal{A})$. Set $E := \cofib[A \to D]$. Then the map $D \to E$ is a projection with kernel in $\pst(\cal{A})$, hence a morphism in $\pst(\cal{C})^{\pst(\cal{A})}_{\pr}$, and the composite $C \to D \to E$ is also a projection with kernel in $\pst(\cal{A})$, as can be seen from the exact sequence
\[ A \to \cofib[C \to D] \to \cofib[C \to E] \]
and the fact that $A \to \cofib[C \to D]$ is a projection in $\pst(\cal{A})$. We conclude that $\iota$ satisfies Assumption~\eqref{item:supply}, and so the proof is complete.
\end{proof}

\begin{proof}[Proof of Proposition~\ref{prop:weak-equivalence}]
We prove the claim for $\cal{C}^{\cal{A}}_{\pr}$, the claim for $\cal{C}^{\cal{A}}_{\inc}$ can be inferred by replacing $\cal{C}$ with its opposite (equipped with the opposite exact structure). Applying Proposition~\ref{prop:pass-to-pst} it will suffice to show that $\pst(\cal{C})^{\pst(\cal{A})}_{\pr} \to \pst(\cal{B})^{\simeq}$ is a weak homotopy equivalence, and applying Proposition~\ref{prop:remove-pr} we may further replace the domain by $\pst(\cal{C})^{\pst(p)}$.
Now the map 
\[ \pst(p)\colon \pst(\cal{C}) \to \pst(\cal{B})\] 
is a Keller projection (Corollary~\ref{cor:heartwise-keller}) and so in particular exhibits $\pst(\cal{B})$ as the localisation of $\pst(\cal{C})$ by the collection of maps belonging to $\pst(\cal{C})^{\pst(p)}$. At the same time, since $\pst(\cal{C})$ has finite colimits and these are preserved by $\pst(p)$ we see that the collection of maps belonging to $\pst(\cal{C})^{\pst(p)}$ is closed under cobase change. The $\infty$-category $\pst(\cal{C})$ is then an $\infty$-category with weak equivalences and cofibrations in the sense of \cite[Definition 7.4.12]{Cis19}, where all maps are cofibrations and the weak equivalences are the maps belonging to $\pst(\cal{C})^{\pst(p)}$. Applying 
\cite[Corollary 7.6.9]{Cis19} we now conclude that the map
\[ \pst(\cal{C})^{\pst(p)} \to \pst(\cal{B})^{\simeq} \]
is a weak homotopy equivalence, as desired.
\end{proof}

\begin{lemma}\label{lem:S-n-exact}
	Suppose given a Keller sequence $\cal{A} \to \cal{C} \to \cal{B}$. Then for each $q \ge 0$, the induced sequence $\rS_q(\cal{A}) \to \rS_q(\cal{C}) \to \rS_q(\cal{B})$ on $\rS$-constructions is also Keller.
\end{lemma}
\begin{proof}
This is a special case of Proposition~\ref{prop:reedy-keller}, as we can identify $\rS_q(-) \simeq \fun^{\reedy}(\Delta^q,-)$ with the induced exact structure.
\end{proof}

\begin{proof}[Proof of Theorem~\ref{thm:localisation}]
Given an extension-closed full inclusion $\cal{D} \subseteq \cal{E}$ of exact $\infty$-categories, and $[n] \in \bbDelta$ let us denote by $P_n(\cal{E};\cal{D}) \subseteq \fun(\Delta^n,\cal{E})$ the full subcategory spanned by those sequences $e_0 \to ... \to e_n$ such that each $e_i \to e_{i+1}$ is an inclusion with cofibre in $\cal{D}$. We note that $P_n(\cal{E};\cal{D})$ is closed under extensions in $\fun(\Delta^n,\cal{E})$ and hence inherits an exact structure. We also note that for every $\rho\colon [n] \to [m]$ in $\bbDelta$ the restriction functor $\fun(\Delta^m,\cal{E}) \to \fun(\Delta^n,\cal{E})$ sends $P_m(\cal{E};\cal{D})$ to $P_n(\cal{E};\cal{D})$, so that the various $P_n(\cal{E};\cal{D})$ assemble to give a simplicial object $P_\bullet(\cal{E};\cal{D})$. Given a Keller sequence
\[ \cal{A} \xrightarrow{i} \cal{C} \xrightarrow{p} \cal{B} \]
we have an induced fibre sequence
\[ P_\bullet(\cal{A};\cal{A}) \to P_\bullet(\cal{C};\cal{A}) \to P_\bullet(\cal{B};0) \]
of simplicial objects in $\exact$. We note that at $\bullet=0$ this is just the original Keller sequence. Applying K-theory and taking geometric realizations we hence obtain a commutative diagram
\[
\begin{tikzcd}
\K(\cal{A}) \ar[r]\ar[d] & \K(\cal{C}) \ar[r]\ar[d] & \K(\cal{B}) \ar[d] \\
{|\K(P_\bullet(\cal{A};\cal{A}))|} \ar[r] & {|\K(P_\bullet(\cal{C};\cal{A}))|} \ar[r] & {|\K(P_\bullet(\cal{B};0)|}
\end{tikzcd}
\]
in $\Sp$.
We now prove the following three claims, which together imply that the top row is an exact sequence in $\Sp$:
\begin{enumerate}
\item\label{item:right-most}
The rightmost vertical map is an equivalence.
\item\label{item:left-square}
The left square is cocartesian.
\item\label{item:bottom-row}
The bottom row is a cofibre sequence.
\end{enumerate}
The first statement is clear, since the simplicial object $P_\bullet(\cal{B};0)$ is homotopy constant. For the second statement, note that the $\infty$-category $\exact$ is semi-additive, and hence we can take the Bar construction $B_\bullet(\cal{C};\cal{A})$ of the exact functor $i\colon \cal{A} \to \cal{C}$, which by definition is the simplicial object obtained by left Kan extending the diagram $0 \leftarrow \cal{A} \xrightarrow{i} \cal{C}$
along the diagram $[0] \xleftarrow{d_0} [1] \xrightarrow{d_0} [0]$ in $\bbDelta^{\op}$. Explicitly, $B_n(\cal{C};\cal{A})$ is given by $\cal{C} \oplus \cal{A} \oplus ... \oplus \cal{A}$ with the summand $\cal{A}$ appearing $n$ times. We have a map of simplicial objects
\[ B_\bullet(\cal{C};\cal{A}) \to P_\bullet(\cal{C};\cal{A}) \]
which at the $n$'th level is the exact functor
\[ f_n\colon \cal{C} \oplus \cal{A} \oplus ... \oplus \cal{A} \to P_n(\cal{C};\cal{A}) \quad\quad (c,a_1,...,a_n) \mapsto [c \to c\oplus a_1 \to c\oplus a_1\oplus a_2 \to ... \to c \oplus a_1\oplus ...\oplus a_n] .\]
We claim that the exact functor $f_n$ induces an equivalence on K-theory for every $n$. Indeed, each $f_n$ admits a retraction 
\[ g_n\colon P_n(\cal{C};\cal{A}) \to \cal{C} \oplus \cal{A} \oplus ... \oplus \cal{A} \quad\quad [c_0 \xrightarrow{i_1} ... \xrightarrow{i_n} c_n] \mapsto (c_0,\cofib(i_1),...,\cofib(i_n)), \]
though, to avoid confusion, let us point out that these retractions do not assemble to a map of simplicial objects. Nonetheless, $g_n\circ f_n = \id$, while the functor $\K(P_n(\cal{C};\cal{A})) \to \K(P_n(\cal{C};\cal{A}))$ induced by $f_n \circ g_n$ is homotopic to the identity by additivity of $\K$, as witnessed by the natural filtration 
\[ 0 \mono [c_0 = c_0 = ... = c_0] \mono (c_0 \xrightarrow{i_1} c_1 = c_1 ... = c_1) \mono ... \mono (c_0 \xrightarrow{i_1} c_1 \xrightarrow{i_2} ... \xrightarrow{i_n} c_n) \]
in $P_n(\cal{C};\cal{A})$ whose associated graded pieces are $[c_0 = ... = c_0], [0 \to \cofib(i_1) = ... = \cofib(i_1)], ... [0 \to ... \to 0 \to \cofib(i_n)]$, the direct sum of which is $f_ng_n(c_0 \to ... \to c_n)$, see Remark~\ref{rem:additive-S}. It hence follows that $f_n$ is indeed an equivalence on K-theory, as claimed. To prove~\eqref{item:left-square} above it will hence suffice to show that the square
\[
\begin{tikzcd}
\K(\cal{A}) \ar[r]\ar[d] & \K(\cal{C}) \ar[d]  \\
{|\K(B_\bullet(\cal{A};\cal{A}))|} \ar[r] & {|\K(B_\bullet(\cal{C};\cal{A}))|}  
\end{tikzcd}
\]
is cocartesian in $\Sp_{\geq 0}$. Now since $\K$-theory preserves direct sums it sends the Bar construction of $\cal{A} \to \cal{C}$ to the Bar construction of $\K(\cal{A}) \to \K(\cal{C})$ in $\Sp_{\geq 0}$ and the Bar construction of $\cal{A} \to \cal{A}$ to the Bar construction of $\K(\cal{A}) \to \K(\cal{A})$. But Bar constructions are left Kan extensions, so their geometric realizations are just the cofibres of the arrows from which they are generated. In other words, the desired statement is equivalent to saying that the square
\[
\begin{tikzcd}
\K(\cal{A}) \ar[r]\ar[d] & \K(\cal{C}) \ar[d]  \\
\cofib[\K(\cal{A}) \to \K(\cal{A})] \ar[r] & \cofib[\K(\cal{A}) \to \K(\cal{C})]  
\end{tikzcd}
\]
is cocartesian, which is clear. It is hence left to prove~\eqref{item:bottom-row}. Now the above argument also shows that $|\K(P_\bullet(\cal{A};\cal{A}))| \simeq |\K(B_\bullet(\cal{A};\cal{A}))| \simeq 0$, and so the desired claim is equivalent to the map
\[ {|\K(P_\bullet(\cal{C};\cal{A}))|} \to {|\K(P_\bullet(\cal{B};0)|} \]
being an equivalence. For this, it will suffice to show that for each $n$ the map
\[ |S_n(P_\bullet(\cal{C};\cal{A}))^{\simeq}| \to |S_n(P_\bullet(\cal{B};0))^{\simeq}| \]
is an equivalence of spaces. We now observe that we can commute $S_n$ and $P_\bullet$, that is, the above map can be rewritten as 
\[ |P_\bullet(S_n(\cal{C});S_n(\cal{A}))^{\simeq}| \to |P_\bullet(S_n(\cal{B});0)^{\simeq}| .\]
This map, in turn, can be identified with the map induced on geometric realizations by the functor
\[ S_n(\cal{C})^{S_n(\cal{A})}_{\inc} \to S_n(\cal{B})^{\simeq}. \]
It will hence suffice to show that this functor is a weak homotopy equivalence. Indeed, this follows from Proposition~\ref{prop:weak-equivalence} applied to the induced Keller sequence on $S_n$ established in Lemma~\ref{lem:S-n-exact}.
\end{proof}

\subsection{The Gillet-Waldhausen theorem}\label{sec:resolution}

In this section we demonstrate how Proposition~\ref{prop:verdier} applies to give a generalization of the Gillet-Waldhausen theorem to all finitary (that is, filtered colimits preserving) Keller localising functors:

\begin{theorem}\label{thm:gillet-waldhausen}
Let $\cal{A}$ be an additive $\infty$-category with finite limits and filtered colimits and let $\cal{F}\colon \exact \to \cal{A}$ be a finitary Keller localising functor. Then the induced map
\[ \cal{F}(\cal{C}) \to \cal{F}(\st(\cal{C})) \]
is an equivalence for any exact $\infty$-category $\cal{C}$.
\end{theorem}

\begin{corollary}\label{cor:universal}
The algebraic $\K$-theory functor $\K\colon \exact \to \Sp$ is the initial Keller localising functor under $\Sigma^{\infty}(-)^{\simeq}$.
\end{corollary}

While the statement of Corollary~\ref{cor:universal} is pleasing from the point of view of the present paper, we in fact expect $\K\colon \exact \to \Sp$ to be initial already as an additive functor under $\Sigma^{\infty}(-)^{\simeq}$.

\begin{proof}[Proof of Corollary~\ref{cor:universal}]
Since $\exact$ is compactly generated (Corollary~\ref{cor:compactly-generated}) the inclusion of finitary functors $\exact \to \Sp$ inside all functors admits a right adjoint, sending $\cal{G}\colon \exact \to \Sp$ to the left Kan extension of its restriction to the full subcategory of compact objects. It will hence suffice to show that $\K$ is the initial finitary Keller localising functor under $\Sigma^{\infty}(-)^{\simeq}$. But by Theorem~\ref{thm:gillet-waldhausen} all finitary Keller localising functors $\exact \to \Sp$ are right Kan extended from their restriction to $\catex \subseteq \exact$. In addition, any Keller localising functor on $\exact$ restricts to a Verdier localising functor on $\catex$. The desired result hence follows from the universal property of algebraic K-theory as a functor on $\catex$, see~\cite{BGT13} or~\cite[Theorem 5.1]{HLS24}.
\end{proof}

The remainder of this subsection is devoted to the proof of Theorem~\ref{thm:gillet-waldhausen}.
Fix a stable category with a bounded heart structure $(\cal{C}_{\ge 0},\cal{C}_{\le 0})$, and write as usual $\cal{C}_{[a, b]} := \Sigma^a\cal{C}_{\ge 0} \cap \Sigma^b\cal{C}_{\le 0}$ for integers $a, b$. For $n\ge 1,$ consider the extension-closed full subcategory $\cal{E}_n \subseteq S_2(\cal{C}_{[0,n]})$ spanned by those exact sequences $x \mono y \epi z$ in $\cal{C}_{[0,n]}$ with $x, y \in \cal{C}_{[0,n-1]}$, considered as an exact subcategory equipped with the induced exact structure. Explicitly, a map $(x\mono y \epi z) \to (x' \mono y' \epi z')$ between exact sequences in $\cal{E}_n$ is a projection if it is so pointwise in $\cal{C}_{[0,n]}$, and $\fib(y \epi y') \in \cal{C}_{[0,n-1]}$, and dually for inclusions. There are exact subcategory inclusions $i_0,i_1:\cal{C}_{[0,n-1]} \to \cal{E}_n$ given respectively by $x \mapsto (x \eeq x \to 0), (0 \to x \eeq x)$, and an exact inclusion $j:\cal{C}_{[0,n-1]} \to \cal{E}_n$ given by $x \mapsto (x \to 0 \to \Sigma x)$.

Recall from \cite[Theorem 2.9]{Sau23} that the inclusion $\cal{C}_{[0,n-1]} \subseteq \cal{C}_{[0,n]}$ is \emph{resolving}, in the following sense:
\begin{enumerate}[leftmargin=*, label=(\roman*)]
	\item any object of $\cal{C}_{[0,n]}$ admits a projection from some object of $\cal{C}_{[0,n-1]}$;
	\item given an exact sequence $x \mono y\epi z$ in $\cal{C}_{[0,n]}$ with $y \in \cal{C}_{[0,n-1]}$, then necessarily $x \in \cal{C}_{[0,n-1]}$.
\end{enumerate}
In addition, any map in $\cal{C}_{[0,n-1]}$ is an inclusion when considered in $\cal{C}_{[0,n]}$. We hence see that forgetting $z$ gives an equivalence of exact $\infty$-categories $\cal{E}_n \simeq \mathrm{Ar}(\cal{C}_{[0,n-1]})$ to the arrow category of $\cal{C}_{[0,n-1]}$ with its pointwise exact structure.
\begin{proposition}\label{cofibre}
	The sequence
	\[
		\cal{C}_{[0,n-1]} \xto{i_0} \cal{E}_n \xto{\ev_2} \cal{C}_{[0,n]}
	\]
	is a Keller sequence, where the second functor is given by the evaluation $(x \mono y \epi z) \mapsto z$.
\end{proposition}
\begin{proof}
The inclusion $i_0$ is Keller by Example~\ref{ex:arrow}.
By Proposition \ref{prop:verdier}, it suffices to show that the functor $\ev_2:\cal{E}_n \to \cal{C}_{[0,n]}$ coincides with the Dwyer-Kan localisation at the projections in $\cal{E}_n$ with fibre in $\cal{C}_{[0,n-1]}$. Since $\ev_2$ certainly inverts maps in $w := (\cal{E}_n)_\mathrm{pr}^{\cal{C}_{[0,n-1]}}$ and is essentially surjective since $\cal{C}_{[0,n-1]} \subseteq \cal{C}_{[0,n]}$ is resolving, it remains to show that the induced functor $\cal{E}_n[w^{-1}] \to \cal{C}_{[0,n]}$ is fully faithful. For $X$ an object of $\cal{E}_n$, write $w_{/X} \subseteq (\cal{E}_n)_{/X}$ for the full subcategory of the slice on those projections $W \epi X$ with fibre in $\cal{C}_{[0,n-1]}$. Given $X_i := (x_i \mono y_i \epi z_i) \in \cal{E}_n$ for $i=0,1$, consider the natural map
	\[
		\underset{w_{/X_0}^{\op}}\colim\ \Map_{\cal{E}_n}(-,X_1) \to \Map_{\cal{C}_{[0,n]}}(z_0,z_1).
	\]
	By universality of colimits, it suffices to show for each map $f:z_0 \to z_1$ that the functor
	\[
		\phi_f:w_{/X_0}^{\op} \to \cal{S}, \quad (W \epi X_0) \mapsto \Map_{\cal{E}_n}(W, X_1) \times_{\Map_{\cal{C}_{[0,n]}}(z_0,z_1)} \{f\}
	\]
	has contractible colimit, and by \cite[Corollary 3.3.4.6]{HTT}, we may show that the unstraightening $\mrm{Un}(\phi_f)$ is weakly contractible. But $\mrm{Un}(\phi_f)$ is equivalent to the $\infty$-category of spans of the form
	\[\begin{tikzcd}
		x_0 \ar[r, rightarrowtail] & y_0 \ar[r, twoheadrightarrow] & z_0 \\
		x' \ar[u, twoheadrightarrow] \ar[r, rightarrowtail] \ar[d] & y' \ar[r, twoheadrightarrow] \ar[u, twoheadrightarrow] \ar[d] & z_0 \ar[u,equal] \ar[d,"f"] \\
		x_1 \ar[r, rightarrowtail] & y_1 \ar[r, twoheadrightarrow] & z_1,
	\end{tikzcd}\]
	with $\fib(y' \epi y_0) \in \cal{C}_{[0,n-1]}$.
	Consider the pullback
	$Q := \lim(y_1 \epi z_1 \leftarrow y_0)$ where the map on the right is the composite $y_0 \to z_0 \xrightarrow{f} z_1$.  
	We note that $Q$ is an extension of $y_0$ by $x_1$, and hence lies in $\cal{C}_{[0,n-1]}$, and moreover $\fib(Q \epi z_0)$ is an extension of $x_0$ by $x_1$, and hence lies in $\cal{C}_{[0,n-1]}$ as well. The object
	\[\begin{tikzcd}
		x_0 \ar[r, rightarrowtail] & y_0 \ar[r, twoheadrightarrow] & z_0 \\
		x' \ar[r, rightarrowtail] \ar[u, twoheadrightarrow] \ar[d] & Q \ar[r, twoheadrightarrow] \ar[d] \ar[u, twoheadrightarrow] & z_0 \ar[u, equal] \ar[d, "f"] \\
		x_1 \ar[r, rightarrowtail] & y_1 \ar[r, twoheadrightarrow] & z_1,
	\end{tikzcd}\]
	is then easily seen to be terminal in $\mrm{Un}(\phi_f)$, so that the latter is weakly contractible, as claimed.
\end{proof}

\begin{corollary}\label{cor:resolution}
The inclusion $\cal{C}_{[0,n-1]} \to \cal{C}_{[0,n]}$ induces an equivalence on $\K$-theory, and in fact on any Keller localising functor. Consequently, the inclusion
\[ \cal{C}^{\heartsuit} \to \cal{C} \]
induces an equivalence on $\K$-theory, and in fact on any finitary 
Keller localising functor.
\end{corollary}
\begin{proof}
Write $\ev_0$, $\ev_1:\cal{E}_n \to \cal{C}_{[0,n-1]}$ for the evaluations $(x \mono y \epi z) \mapsto x$, $y$. Then the composite
\[
	\cal{C}_{[0,n-1]}^{\times 2} \xto{j \oplus i_1} \cal{E}_n \xto{\ev_0,\ev_1} \cal{C}_{[0,n-1]}^{\times 2}
\]
is the identity on $\cal{C}_{[0,n-1]}^{\times 2}$, and the composite
\[
	\cal{E}_n \xto{\ev_0,\ev_1} \cal{C}_{[0,n-1]}^{\times 2} \xto{j \oplus i_1} \cal{E}_n
\]
sends $(x \mono y \epi z)$ to $(x \xto{0} y \xto{(\id,0)} y \oplus \Sigma x)$. Its action on K-theory is then homotopic to the identity by virtue of additivity and the diagram
\[\begin{tikzcd}
	0 \ar[r] \ar[d] & y \ar[r, equal] \ar[d, equal] & y \ar[d] \\
	x \ar[r, rightarrowtail] \ar[d, equal] & y \ar[r, twoheadrightarrow] \ar[d] & z \ar[d] \\
	x \ar[r] & 0 \ar[r] & \Sigma x,
\end{tikzcd}\]
and so we conclude that $(\ev_0,\ev_1)$ induces an equivalence on K-theory. At the same time, the composite
\[ \cal{C}_{[0,n-1]} \oplus \cal{C}_{[0,n-1]} \xto{i_0 \oplus i_1} \cal{E}_n \xto{(\ev_0,\ev_1)} \cal{C}_{[0,n-1]} \oplus \cal{C}_{[0,n-1]} \]
is given by $(x,y) \mapsto (x, x \oplus y)$, and hence induces an equivalence on K-theory by direct sum preservation of the latter. We conclude that $i_0 \oplus i_1$ is an equivalence on K-theory. 
Finally, consider the commutative diagram
\[
\begin{tikzcd}[column sep = 50pt]
\cal{C}_{[0,n-1]} \ar[r,"{x \mapsto (x,0)}"]\ar[d,equal] & \cal{C}_{[0,n-1]} \oplus \cal{C}_{[0,n-1]} \ar[d,"i_0\oplus i_1"]\ar[r,"{(x,y) \mapsto y}"] & \cal{C}_{[0,n-1]} \ar[d] \\
\cal{C}_{[0,n-1]} \ar[r,"i_0"] & \cal{E}_n \ar[r,"\ev_2"] & \cal{C}_{[0,n]}
\end{tikzcd}
\]
where the top row is clearly a Keller sequence and the bottom row is a Keller sequence by Proposition~\ref{cofibre}. Since the left-most and middle vertical maps are equivalences on K-theory we conclude from Theorem~\ref{thm:localisation} that the rightmost vertical map is an equivalence, as desired.
\end{proof}

\begin{proof}[Proof of Theorem~\ref{thm:gillet-waldhausen}]
Combine Corollary~\ref{cor:resolution} with \cite[Proposition 5.4]{SW25}.
\end{proof}

\bibliographystyle{amsalpha}
\bibliography{keller-bib}

\end{document}